\documentclass[12pt,a4paper,reqno]{amsart}
\usepackage{mathrsfs}
\usepackage{calc}
\usepackage{cite}

\usepackage{amssymb,amsmath,latexsym, amsthm,enumerate,amsfonts,graphicx}
\usepackage[mathscr]{euscript}

\input amssym.def
\input amssym

\def\dj{d\kern-0.4em\char"16\kern-0.1em}
\def\Dj{\mbox{\raise0.3ex\hbox{-}\kern-0.4em D}}

\def\be{\begin{equation}}
\def\ee{\end{equation}}
\def\bena{\begin{eqnarray*}}
\def\ena{\end{eqnarray*}}

\def\t{\tau}
\def\s{\sigma}

\def\suml{\sum\limits}

\def\dss{\displaystyle}

\newcommand{\WF}{\operatorname{WF}}

 \def\D{\mathcal{D}}
 \def\E{\mathcal{E}}
 \def\Rd{\mathbf{R}^d}
 \def\Z{\mathbf{Z}_+}

\def\N{\mathbf{N}}

\def\lf {\lfloor}
\def\rf{\rfloor}

\newcommand{\sing}{\operatorname{singsupp}}
\newcommand{\supp}{\operatorname{supp}}
\numberwithin{equation}{section}

\newtheorem{te}{Theorem}[section]
\newtheorem{lema}{Lemma}[section]
\newtheorem{prop}{Proposition}[section]
\newtheorem{cor}{Corollary}[section]

\theoremstyle{definition}

\newtheorem{de}{Definition}[section]
\theoremstyle{remark}
\newtheorem{rem}{Remark}[section]

\begin{document}

%
%
%
%
%
%
%
%
%

\title{\textbf{Beyond Gevrey regularity}}

\author{Stevan Pilipovi\' c}

\address{Department of Mathematics and Informatics,
University of Novi Sad, Novi Sad, Serbia}

\email{stevan.pilipovic@dmi.uns.ac.rs}

\author{Nenad Teofanov}

\address{Department of Mathematics and Informatics,
University of Novi Sad, Novi Sad, Serbia}

\email{nenad.teofanov@dmi.uns.ac.rs}

\author{Filip Tomi\'c}

\address{Faculty of Technical Sciences,
University of Novi Sad, Novi Sad, Serbia}

\email{filip.tomic@uns.ac.rs}
\subjclass{Primary 35A18, 46F05; Secondary 46F10}

\keywords{Ultradifferentiable functions, Gevrey classes,
ultradistributions, wave-front sets}

\date{November, 2015}


\begin{abstract}
We define and study classes of smooth functions which are less
regular than Gevrey functions. To that end we introduce
two-parameter dependent sequences which do not satisfy Komatsu's
condition (M.2)', which implies stability under differential operators within the spaces of
ultradifferentiable functions.
Our classes therefore have particular behavior under the action of
differentiable operators. On a more advanced level, we study
microlocal properties and prove that
$$\WF_{0,\infty}(P(D)u)\subseteq \WF_{0,\infty}(u)\subseteq \WF_{0,\infty}(P(D)u) \cup {\rm Char}(P),$$
where $u$  is a Schwartz distribution, $P(D)$ is a  partial differential operator with constant coefficients
and $\WF_{0,\infty}$ is the wave front set described in terms of new regularity conditions.
For the analysis
we introduce particular admissibility condition for sequences of cut-off functions,
and a new technical tool called enumeration.
\end{abstract}

\maketitle
\section{Introduction}\label{sec-0}

\par

We propose new regularity conditions for smooth
functions which are weaker than the Gevrey regularity conditions.
Instead of the Gevrey sequence $ \{ p!^t \}_{p \in {\mathbf N}} $,
determined by parameter $ t>1$, we observe two-parameter
dependent sequences of the form $ \{ p^{\t p^\s} \}_{p \in {\mathbf N}} $,
with $ \t > 0 $ and $ \s > 1.$ When $\s = 1$ and $ \t > 1 $ we recapture the Gevrey
regularity as well as the analytic regularity for  $\s = 1$ and $ \t = 1 $.

Gevrey classes were initially introduced for the study of
regularity properties of the fundamental solution of the heat
operator, cf. \cite{Gevrey}, and thereafter used to describe regularities stronger than smoothness
and weaker than analyticity. In particular, it turned out that the well-posedness of the
Cauchy problem for weakly hyperbolic linear partial differential
equations (PDEs) can be characterized by the Gevrey index $t$, while
the same problem is ill-posed in the class of analytic functions,
cf. \cite{Chen, Rodino} and the references given there. Roughly speaking,
fundamental solution $\phi $ may have $C^{\infty}$-regularity property, which in this paper means that it
is smooth without restrictions to the growth of its derivatives,
$\E_t$-regularity (Gevrey regularity) if $\partial^{\alpha}\phi$  are bounded by $C^{\alpha+1}\alpha!^t$, $\alpha\in {\mathbf N}^d$,
for some $C>0$, $t>1$, and ${\mathcal A}$-regularity if
$\partial^{\alpha}\phi$ are bounded by $C^{\alpha+1}\alpha!$, $\alpha\in {\mathbf N}^d$,
for some $C>0$. Since there is a gap between the Gevrey and $C^{\infty}$-regularity,
new regularity conditions could be useful in local analysis of the solutions of PDEs, which is one motivation
for our work. In particular, our condition describes hypoellipticity property standing between $C^{\infty}$ hypoellipticity and
Gevrey hypoellipticity.

\par

Another motivation comes from microlocal analysis, where the notion of wave-front set plays a crucial role.
We recall that
\be
\label{wavefrontpodskup1}
\WF(u)\subseteq \WF_t(u) \subseteq \WF_A(u)\,, \;\;\; t>1,
\ee
where $u$ is a Schwartz distribution, $\WF$ is the classical $(C^{\infty})$ wave front set,
$\WF_t$ is the Gevrey wave front set, and $\WF_A$ is analytic wave front set,
we refer to Subsection \ref{subsec01} for precise definitions,
and to \cite{ Foland, HermanderKnjiga} for details.
We note that one can find examples of (ultra)distributions for which
the inclusions in \eqref{wavefrontpodskup1} are strict,
and the same holds for other inclusions of wave front sets in this paper.
Extension of \eqref{wavefrontpodskup1}
to Gevrey type ultradistributions is given in
\cite{Rodino} and "stronger" singularities related to $t<1$ are recently treated in
\cite{Pilipovic-Toft}.

Apart from the Gevrey wave front set,
different types of wave front sets that modify the classical wave front set
are introduced in the literature in connection to the equation under investigation, and
we do not intend to survey the definitions here. However, let us briefly mention the Gabor wave front set,
originally defined in \cite{HermanderRad} and further developed in
\cite{RodinoWahlberg}, which is based on microlocal analysis on cones taken with respect to
the whole of the phase space variables. Such approach is recently successfully
applied to the study of Schr\"odinger equations
in \cite{CW, CNR-1, CNR-2, P-SRW, Wahlberg}, see also the references therein.
Note that the Gabor wave front set of a tempered distribution is characterized in terms of rapid decay of its
Gabor coefficients on appropriate set. The idea to use Gabor coefficients and, consequently,
methods of time-frequency analysis and modulation spaces
in the study of wave front sets is introduced in \cite{JPTT, PTToft-01, PTToft-02}, and extended in \cite{CJT-1, CJT-2}
to more general Banach and Fr\'echet spaces.  We refer to \cite{F1,FG1, FS1, FS2} for
details on modulation spaces and their role in time-frequency analysis.
Since versions of Gabor wave front set can be adapted to analytic and Gevrey regularity
(cf. \cite{CS, SW-1, SW-2}) it is natural to
assume that the same holds in the framework of regularity proposed in
this paper, which will be considered by the authors in a separate contribution.

\par

Our approach gives a possibility to define
wave-front sets which detect
singularities that are "stronger" then the classical $C^{\infty}$
singularities and at the same time "weaker" than any  Gevrey type
singularities, and to show that the usual properties (such
as pseudo-local property), valid for wave-front sets quoted in (\ref{wavefrontpodskup1}), hold
also in the context of our new regularity conditions.
More precisely, one of the main results of the paper is the following
(see Section \ref{sec-2} for the definition of $\WF_{\{\t,\s\}}(u)$).

\begin{te} \label{glavnateorema}
Let $\t>0$, $\s>1$, and $u\in \D'(U)$.
Then
\begin{multline}
\label{FundamentalEstimate}
\WF_{\{2^{\s-1}\t,\s\}}(P(D)u)\subseteq \WF_{\{2^{\s-1}\t,\s\}}(u)\\[1ex]
\subseteq
\WF_{\{\t,\s\}}(P(D)u) \cup {\rm Char}(P(D)),
\end{multline}
where $P(D)$ is a partial differential operator of order $m$ with constant coefficients
and $ {\rm Char}(P(D))$ is its characteristic set.
\end{te}

In fact,  the result of Theorem \ref{glavnateorema} holds true when
$ P(D)=\sum_{|\alpha|\leq m} a_{\alpha} (x) D^{\alpha}$, where
$ a_{\alpha} (x) \in \E_{\{\t,\s\}} (\Rd)$ (see Section \ref{sec-1} for the definition).
This extension requires nontrivial modifications of the proof of Theorem \ref{glavnateorema}
and will be given in another paper.
In particular, to handle the approximate solution (see Section \ref{sec-3})
one should prove and use inverse closedness property of the corresponding algebra,
cf. \cite{K}.

We refer to \eqref{characteristic set} for the definition of  $  {\rm Char}(P(D)) $ and recall that
if $  {\rm Char}(P(D)) = \emptyset $ then   $P(D)$ is called hypoelliptic.

In particular, with
$
\WF_{0,\infty}(u)=\bigcup_{\s>1}\bigcap_{\t>0}\WF_{\{\t,\s\}}(u)$
we have:
\begin{cor}
\label{FundamentalEstimateCor}
Let $u\in \D'(U)$ and $P(D)$ be a partial differential operator of order $m$ with constant coefficients.
Then
\be
\WF_{0,\infty}(P(D)u)\subseteq \WF_{0,\infty}(u)\subseteq \WF_{0,\infty}(P(D)u) \cup {\rm Char}(P(D)).
\ee
\end{cor}

For the proof of Theorem \ref{glavnateorema} we perform a careful analysis of sequences of cut-off test functions which
lead to specific {\em admissibility condition}. Moreover, we introduce a simple procedure called {\em enumeration}
which is quite useful for the description of asymptotic behavior in microlocalization.
In short, enumeration of a sequence "speeds up" or "slows down"
the decay estimates of single terms
while preserving the asymptotic behavior of the whole sequence.

\par

Different values of parameters $\t > 0 $ and $ \s >1 $ define
different local regularity conditions which in turn implies that in many situations we obtain {\em strict
inclusions} between the corresponding wave front sets. In particular,
$\WF (u) $ is, in general,  a strict subset of the intersections of our wave front sets,
while the intersection of
the Gevrey wave front sets, $ \cap_{t>1}\WF_t$ contains the union of our wave front sets
as a strict subset, see Corollary \ref{strict inclusions}.

\par

We note that our wave front sets are different from $WF_L$
introduced in
\cite[Chapter 8.4]{HermanderKnjiga}
with respect to $C^L$ regularity classes defined by
an increasing sequence of positive numbers such that $p \leq L_p $ and $L_{p+1}\leq C L_p$,
for some $C>0$ and for every $ p \in {\mathbf N}$. When $ L_p = (p+1)^t, $ $ t>1 $, $C^L $ is the Gevrey class.
However, our defining sequence $ \{ p^{\t p^\s} \}_{p \in {\mathbf N}} $
gives $L_p = p^{\t p^{\s-1}}$,
which does not satisfy $L_{p+1}\leq C L_p$, $ p \in {\mathbf N}$,
for any choice of $\t>0$, $\s>1$. Therefore our approach describes another type of regularity than
$C^L$ regularity.

\par

The paper is organized as follows. In Section \ref{sec-1} we observe
sequences of the form $ \{ p^{\t p^\s} \}_{p \in {\mathbf N}} $, $ \t > 0 $
and $ \s > 1$, which  do not satisfy Komatsu's property
(M.2)' (stability under differentiation) which is the basic one
in the theory of ultradifferentiable functions, cf.
\cite{Komatsuultra1}. Next, we use such sequences to
define spaces of ultradifferentiable functions of
regularity weaker than the Gevrey regularity, and study their main
properties. In particular, we discuss stability under the action
of ultradifferentiable operators.

In Section \ref{sec-2} we review the most common local regularity conditions and
wave-front sets of (ultra)distributions.
We introduce the notion of {\em enumeration} to motivate
the definition of wave-front sets with respect to the
regularity introduced in Section \ref{sec-1}.
Due to specific properties of our defining sequences, we had to modify H\"ormander's construction
from \cite{HermanderKnjiga} by introducing a new \emph{admissibility condition}
for sequences of cut-off functions used in the microlocalization.
Next, we describe local regularity via decay estimates on the
Fourier transform side (Propositions \ref{dovoljanUslov} and
\ref{potrebanUslov}) and discuss singular supports
of (ultra)distributions.

Finally, in Section \ref{sec-3} we prove Theorem \ref{glavnateorema}.
Although we follow the general idea
of the proof of
\cite[Theorem 8.6.1]{HermanderKnjiga}
we present here a detailed proof
since our approach brings nontrivial changes and modifications into it.

\par

We remark that some preliminary results of our investigations
are given in \cite{PTT-01}, where test function spaces
for Roumieu type ultradistributions  were considered.

\subsection{Notation} \label{subsec01}
Sets of numbers are denoted in a usual way, e.g. ${\bf N}$ (resp. $\Z$) denotes the set of nonnegative ( resp. positive) integers.
For $x\in {\bf R}_+$ the floor and the ceiling
functions are denoted by $\lf x \rf:=\max\{m\in
\N\,:\,m\leq x\}$ and $\lceil x \rceil := \min\{m\in
\N\,:\,x\leq m\}$. For a multi-index
$\alpha=(\alpha_1,\dots,\alpha_d)\in {\bf N}^d$ we write
$\partial^{\alpha}=\partial^{\alpha_1}\dots\partial^{\alpha_d}$ and
$|\alpha|=|\alpha_1|+\dots |\alpha_d|$. We will often use Stirling's formula:
$$
N!=N^N e^{-N}\sqrt{2\pi N}e^{\theta_N \over 12N},
$$
for some $0<\theta_N<1$, $N\in\N \setminus 0.$ By $C^{m}(K)$, $m\in \N$,
we denote the Banach space of $m$-times continuously differentiable
functions on a compact set $K\subset\subset U$ with smooth boundary, where $U
\subseteq \Rd$ is an open set, $C^{\infty}(K)$ denotes the
set of smooth functions on $K$, $C_K^{\infty}$ are smooth functions supported by $K$, and
${\mathcal A}(U)$ denotes the space of analytic functions on $U$.
The closure of $ U\subset \Rd$ is denoted by $ \overline{U}$.
A conic neighborhood of $\xi_0 \in \Rd  \setminus 0$  is an open cone
$\Gamma \subset \Rd $ such that  $ \xi_0 \in \Gamma$.
The convolution is given by
$(f*g)(x)=\int_{\Rd}f(x-y)g(y) dy$, whenever the integral makes sense.
The Fourier transform $\widehat{f} $ of a  locally integrable function $f$ is normalized to be
${\mathcal F} (f) (\xi) = \widehat{f}(\xi)=\int_{{\Rd}}f(x)e^{-2\pi i x\xi}dx$, $\xi \in \Rd$,
and the definition extends to distributions by duality. Open ball of radius $r$, centered at $x_0\in \Rd$ is denoted by $B_r(x_0)$.

\par

For locally convex topological spaces   $X$ and $Y$,
$X\hookrightarrow Y$ means that $X$ is dense in $Y$ and that the identity mapping
from $X$ to $Y$ is continuous, and we use
$ \varprojlim $  and $\varinjlim $ to denote the projective and inductive limit topologies
 respectively.
By $X'$ we denote the strong dual of $X$ and by
$\langle \cdot, \cdot \rangle_{X}$ the dual pairing between $X$ and
$X'$.

\par

We will observe  $ P(D)=\sum_{|\alpha|\leq m}a_{\alpha} D^{\alpha}$
partial differential operators of order $m$ with constant coefficients.
Then $ P(\xi)=\sum_{|\alpha|\leq m}a_{\alpha} \xi^{\alpha}$, $\xi\in \Rd\backslash\{0\}$,
is the symbol of $P(D)$ and
$\ P_m( \xi)=\sum_{|\alpha|= m}a_{\alpha}{\xi}^{\alpha}$, $\xi\in \Rd\backslash\{0\}$,
is its principal symbol. The characteristic set of $P(D)$ is then given by
\be \label{characteristic set}
{\rm Char}(P(D))=\{\xi \in \Rd\backslash\{0\}\,|\,P_{m}(\xi)=0\}.
\ee
Let $ x_0 \in U $ and  $\xi_0 \not \in {\rm Char}(P)$.
Then there is an compact neighborhood $K\subset U $ of $x_0$ and
a conic neighborhood $\Gamma$ of $\xi_0 $ such that $ P_{m}(\xi)\not=0$
for all $(x, \xi) \in K \times \Gamma$.
Moreover, there exist $ C_1, C_2 >0$ such that
\be
\label{OcenaSaDonjeStrane}
C_1 |\xi|^m \leq P_m(\xi)\leq C_2|\xi|^{m},\quad (x, \xi) \in K\times  \Gamma.
\ee

\par

As usual,  $\D'(U)$ stands for Schwartz distributions, and $\E'(U)$ for compactly supported distributions.
We refer to \cite{Komatsuultra1} for the definition and detailed study of different classes
of ultradifferentiable functions and their duals, and to Remark \ref{EkvivalentneNorme} for the definition of Gevrey classes $\E_{t}(U)$, $\D_{t}(U)$, $t>1$.

Let  $t>1$ and $(x_0,\xi_0)\in U\times \Rd\backslash\{0\}$ and $u\in \D '(U)$.
Then the {\em Gevrey wave front set} $WF_t(u)$
can be defined as follows: $(x_0,\xi_0)\not\in WF_t (u)$ if and only if there exists an open neighborhood
$\Omega$ of $x_0$, a conic neighborhood $\Gamma$ of $\xi_0$ and a bounded sequence
$u_N\in \E'(U)$, such that $u_N=u$ on $\Omega$ and
\be \label{Gevreycondition}
|\widehat u_N(\xi)|\leq A\, \frac{h^{N}N!^t }{|\xi|^{ N }}, \quad N\in { \Z},\,\xi\in\Gamma,
\ee
for some $A,h>0$. In fact, we may take $ u_N = \phi u $ for some $\phi\in \D_{t}(U)$ which is equal to $1$
in a neighborhood of $x_0$.
If $t=1$ in \eqref{Gevreycondition}, then the corresponding wave-front set is  called the {\em analytic wave front set}
and denoted by  $WF_A(u)$.
We refer to \cite{HermanderKnjiga, Rodino, Foland} for
the classical  wave-front set.

\section{Regularity classes $\E_{\t,\s}$} \label{sec-1}

In this section we first observe sequences
$M_p^{\t,\s}=p^{\t p^{\s}}$, $p\in \N $, where $\t>0$ and  $\s>1$, and
list their basic properties in Subsection \ref{SekcijaOsobineNiza}.
The flexibility obtained by introducing the two-parameter dependence
enables us to introduce and study smooth  functions which are less regular than the Gevrey functions,
see Subsection \ref{klase}.
In Subsection \ref{svojstva} the action of ultradifferentiable operators
on such classes is studied.

\subsection{The defining sequence $M_p^{\t,\s}$}\label{SekcijaOsobineNiza}
Basic properties of our defining sequences are
given in the following lemma. We refer to \cite{PTT-01} for the proof.

\begin{lema} \label{osobineM_p_s}
Let $\tau>0$, $\s>1$ and $M_p^{\tau,\s}=p^{\tau p^{\s}}$, $p\in \Z$, $M_0^{\tau,\s}=1$.
Then the following properties hold:

$(M.1)$ $(M_p^{\tau,\s})^2\leq M_{p-1}^{\tau,\s}M_{p+1}^{\tau,\s}$, $p\in \Z$, \medskip

$\widetilde{(M.2)'}$
$M_{p+1}^{\tau,\s}\leq C^{p^{\s}}M_p^{\tau,\s}$, for some  $C> 1$, $p\in \N$,  \medskip

$\widetilde{(M.2)}$ $M_{p+q}^{\tau,\s}\leq
C^{p^{\s} + q^{\s}}M_p^{\tau 2^{\s-1},\s}M_q^{\tau 2^{\s-1},\s}$,
$p,q\in \N$, for some $ C>1$.

$(M.3)'$ $ \displaystyle \suml_{p=1}^{\infty}\frac{M_{p-1}^{\t,\s}}{M_p^{\t,\s}} <\infty.$
\end{lema}

If  $\s=1$ then $\widetilde{(M.2)'}$ and $\widetilde{(M.2)}$
are standard Komatsu's conditions $(M.2)'$ and $(M.2)$,
respectively.

We will occasionally use
Stirling's formula
 \begin{equation}
\label{niz1}
\lf p^{\s} \rf!^{\t/\s}\sim  (2\pi)^{\t/(2\s)}p^{\t/2}e^{-(\t/\s)p^{\s}}M_p^{\t,\s} , \quad p\to \infty.
\end{equation}

\subsection{Classes of ultradifferentiable functions} \label{klase}
Let $\tau>0$, $\s>1$, $h>0$, and
$K\subset\subset U$, where $U$ is an open set in $\Rd$. A smooth function $\phi$ on $U$
belongs to the space ${\E}_{\t, {\s},h}(K)$ if there exists $A>0$
such that
$$ \displaystyle |\partial^{\alpha}\phi(x)|\leq A
h^{|\alpha|^{\s}}|\alpha|^{\t|{\alpha}|^{\s}},\quad \alpha \in \N^d, x\in K.
$$
It is a Banach space with the norm given by
\begin{equation} \label{Norma}
\| \phi \|_{{\E}_{\t, {\s},h}(K)}=\sup_{\alpha \in \N^d}\sup_{x\in K}
\frac{|\partial^{\alpha} \phi (x)|}{h^{|\alpha|^{\s}}|\alpha|^{\t|\alpha|^{\s}}}\,,
\end{equation}
and $ \displaystyle
{\E}_{\t_1, {\s_1},h_1}(K)\hookrightarrow {\E}_{\t_2,
{\s_2},h_2}(K)$, $0<h_1\leq h_2,$ $0<\t_1\leq\t_2$, $1<\s_1\leq \s_2.$

Let ${\D}^K_{\t, \s,h}$ be the set of functions in ${\E}_{\t,
\s,h}(K)$ with support contained in $K$. Then, in the topological sense,
we set
\begin{equation}
\label{NewClassesInd} {\E}_{\{\t,
\s\}}(U)=\varprojlim_{K\subset\subset U}\varinjlim_{h\to
\infty}{\E}_{\t, {\s},h}(K),
\end{equation}
\begin{equation}
\label{NewClassesProj} {\E}_{(\t,
\s)}(U)=\varprojlim_{K\subset\subset U}\varprojlim_{h\to 0}{\E}_{\t,
{\s},h}(K),
\end{equation}
\begin{equation}
\label{NewClassesInd2} {\D}_{\{\t,
\s\}}(U)=\varinjlim_{K\subset\subset U} {\D}^K_{\{\t, \s\}}
=\varinjlim_{K\subset\subset U} (\varinjlim_{h\to\infty}{\D}^K_{\t,
\s,h})\,,
\end{equation}
\begin{equation}
\label{NewClassesProj2} {\D}_{(\t,
\s)}(U)=\varinjlim_{K\subset\subset U} {\D}^K_{(\t, \s)}
=\varinjlim_{K\subset\subset U} \varprojlim_{h\to 0}{\D}^K_{\t,
\s,h}.
\end{equation}
We will use abbreviated notation $ \t,\s $ for
$\{\t,\s\}$ or $(\t,\s)$ .
It can be proved that
the spaces ${\E}_{\t, \s}(U)$, ${\D}^K_{\t, \s}$ and ${\D}_{\t, \s}(U)$
are nuclear, cf. \cite{PTT-01}.

\begin{rem} \label{EkvivalentneNorme}
From Lemma \ref{osobineM_p_s} it follows that the norms in \eqref{Norma} can be replaced by
\begin{equation}
\|\phi\|^{\sim}_{{\E}_{\tau, {\s},h}(K)}=\sup_{\alpha \in
\N^d}\sup_{x\in K}\frac{|\partial^{\alpha} \phi
(x)|}{h^{|\alpha|^{\s}}\lf|\alpha|^{\s}\rf!^{\tau/\s}}<\infty,\quad
h>0.
\end{equation} If $ \t > 1 $ and $\s = 1$,
then $ {\E}_{\t, 1}(U)={\E}_{\t}(U)$  are the Gevrey classes
and $\D_{\t,1}(U)=\D_{\t}(U)$ are the corresponding subspaces of compactly supported functions in $\E_{\t}(U)$. When $0<\t\leq 1$ and $ \s = 1$ such spaces are contained in the corresponding spaces of quasianalytic functions. In particular, $\dss
\D_{\t}(U)=\{0\}$ when $0<\t\leq 1$.
\end{rem}

By the Borel Theorem (cf. \cite{HermanderKnjiga, N}), there exists a smooth function $f$ such that
$$f^{(p)}(0)=p^{\tau p^\sigma}, p\in \mathbb Z_+,
$$
and from the Whitney extension theorem we may conclude that ${\E}_{\t, \s}(U) \neq \emptyset$.
However, there does not exist any sequence $(M_p)_p$ of
the Komatsu class so that the corresponding space of ultradifferentiable functions contain $f$.
Moreover, the existence of compactly supported functions in ${\D}_{\t, \s}(U)$
which are not in Gevrey classes ${\D}_t(U)$ for any $t>1$, and
of compactly supported function $\phi\in {\E_{\t,\s}}(U)$
such that $0\leq\phi\leq 1$ and $\int_{\Rd}\phi\,dx=1$ is discussed in
\cite{PTT-01}.

The basic embeddings between the introduced spaces with respect to $\s$ and $\t$ are
given in the following proposition.

\begin{prop}
\label{detectposition} Let $\s_1\geq 1$. Then for every $\s_2>\s_1$
and $\t>0$
\begin{equation}
\label{Theta_S_embedd} \varinjlim_{\t\to \infty}{\E}_{\t,
{\s_1}}(U)\hookrightarrow \varprojlim_{\t\to 0^+} {\E}_{\t,
{\s_2}}(U).
\end{equation} Moreover, if $0<\t_1<\t_2$, then for every $\s\geq 1$ it holds
\be \label{RoumieuBeurling} \E_{\{\t_1,\s\}}(U)\hookrightarrow
\E_{(\t_2,\s)}(U)\hookrightarrow \E_{\{\t_2,\s\}}(U), \ee
and
$$
\varinjlim_{\t\to \infty}{\E}_{\{\t, {\s}\}}(U)= \varinjlim_{\t\to \infty} {\E}_{(\t, {\s})}(U),
\;\;\;
\varprojlim_{\t\to 0^+}{\E}_{\{\t, {\s}\}}(U)= \varprojlim_{\t\to 0^+} {\E}_{(\t, {\s})}(U).
$$

\end{prop}

\begin{proof}
For the proof of \eqref{Theta_S_embedd} we refer to \cite[Proposition 2.1]{PTT-01}.
Since the second embedding in (\ref{RoumieuBeurling}) is trivial, we proceed with the proof of the first one.
Let $\phi\in \E_{\t_1,\s,k}(K)$ for some $k>0$. Since
$$
\dss ||\phi||_{\E_{\t_2,\s,h}(K)}\leq \sup_{\alpha\in \N^d}\frac{k^{|\alpha|^{\s}}|\alpha|^{\t_1 |\alpha|^{\s}}}{h^{|\alpha|^{\s}}|\alpha|^{\t_2 |\alpha|^{\s}}} ||\phi||_{\E_{\t_1,\s,k}(K)},\quad h,k>0,
$$
and
$\dss \sup_{\alpha\in \N^d}\frac{k^{|\alpha|^{\s}}|\alpha|^{\t_1 |\alpha|^{\s}}}{h^{|\alpha|^{\s}}|\alpha|^{\t_2 |\alpha|^{\s}}}
\leq
e^{\frac{\t_2-\t_1}{e \s}(k/h)^{\frac{\s}{\t_2-\t_1}}},
$
then for any given $h>0$ there exists $C>0$ such that
$ \dss
||\phi||_{\E_{\t_2,\s,h}(K)}\leq C ||\phi||_{\E_{\t_1,\s,k}(K)},
$
and the proof is finished.
\end{proof}

We denote the corresponding projective
(when $\t \rightarrow 0^+ $ or when $\s \rightarrow 1^+$) and inductive
(when $\t \rightarrow \infty $  or when  $\s \rightarrow \infty$)
limit spaces  as follows:
$$
\E_{0,\s}(U):=\varprojlim_{\t\to 0^+}\E_{\t,{\s}}(U), \;\;\;
\E_{\infty,\s}(U):=\varinjlim_{\t\to \infty}\E_{\t,{\s}}(U),
$$
$$
\E_{\t,1}(U):=\varprojlim_{\s\to 1^+}\E_{\t,{\s}}(U), \;\;\;
\E_{\t,\infty}(U):=\varinjlim_{\s\to \infty}\E_{\t,{\s}}(U),
$$
\be
\label{BorderLinePresek}
\E_{0,1}(U):= \varprojlim_{\s\to 1^+}\E_{0,\s}(U), \;\;\;
\E_{0,\infty}(U):=\varinjlim_{\s\to \infty}\E_{0,\s}(U),
\ee
\be
\label{BorderLineUnija}
\E_{\infty,1}(U):=\varprojlim_{\s\to 1^+}\E_{\infty,\s}(U), \;\;\;
\E_{\infty,\infty}(U):=\varinjlim_{\s\to \infty}\E_{\infty,\s}(U),
\ee
Then Proposition \ref{detectposition} implies
the following dense embeddings:
\begin{multline}
\label{tildaprop1} \varinjlim_{\t\to\infty}\E_{\t}(U)\hookrightarrow
{\E}_{0, 1}(U)\hookrightarrow {\E}_{\infty,
1}(U)
\\[1ex]
\hookrightarrow \E_{0,\infty}(U) \hookrightarrow
\E_{\infty,\infty}(U) \hookrightarrow C^\infty (U).
\end{multline}

In fact, the first embedding $\dss \varinjlim_{t\to\infty}\E_{\t}(U)\hookrightarrow
{\E}_{0, 1}(U) $ in (\ref{tildaprop1}) follows directly from
Proposition \ref{detectposition} when $\s_2>\s_1=1$.
The embedding $\dss {\E}_{0, 1}(U)\hookrightarrow {\E}_{\infty, 1}(U)$ is obvious.
Fix $\s_1>1$ and let $\s_2>\s_1$. Then for some $\t_0>0$
$$
{\E}_{\t_0, \s_1}(U)
\hookrightarrow {\E}_{0, \s_2}(U)
\hookrightarrow \E_{0,\infty}(U),
$$
where the first embedding follows from (\ref{Theta_S_embedd}) and the last one is trivial. This implies $\E_{\infty,1}(U)\hookrightarrow \E_{0,\infty}(U)$. Since the embeddings
$$
\E_{0,\infty}(U) \hookrightarrow \E_{\infty,\infty}(U)\hookrightarrow C^\infty (U)
$$
are trivial,  \eqref{tildaprop1} is proved.

\subsection{Continuity properties of ultradifferentiable operators on ${\E_{\t,\s}}(U)$} \label{svojstva}
The space ${\E}_{\tau, \s}(U)$ can not be closed
under the action of differential operator $\partial^\alpha $ for any given $\t>0$ and $\s>1$
since then $M_p^{\tau,\s}$ does not satisfy Komatsu's
condition (M.2)'. However, if we consider
$\E_{\infty,\s}(U)$
instead, then $\widetilde{(M.2)}$ provides the continuity of certain ultradifferentiable operators.

\begin{de} Let $\t > 0 $ and $\s \geq 1 $ and let
$a_{\alpha} (x) \in {\E}_{(\t, \s)}(U)$ (resp. $a_{\alpha} (x) \in {\E}_{\{\t,
\s\}}(U)$.
Then
$ \displaystyle
P(x,\partial)=\suml_{|\alpha|=0}^{\infty}a_{\alpha}(x){\partial}^{\alpha}
$
is ultradifferentiable operator of class $(\t,\s)$
(resp. $\{\t,\s\}$) on $U\subseteq \Rd$ if for every $K\subset\subset U$ there exists constant $L>0$ such that for any $h>0$ there exists $A>0$ (resp. for every $K\subset\subset U$ there exists $h>0$ such that for any $L>0$ there exists $A>0$) such that,

\begin{equation}
\label{Operatortausigma}
\sup_{x\in K}|\partial^{\beta}a_{\alpha}(x)|\leq A h^{{|\beta|}^{\s}}|\beta|^{\t{|\beta|}^{\s}}\frac{L^{|\alpha|^{\s}}}{|\alpha|^{\t 2^{\s-1}{|\alpha|}^{\s}}},\quad {\alpha,\beta \in \N^d}.
\end{equation} $P(x,\partial)$ is of the class $\t,\s$ if it is of the class $(\t,\s)$ or $\{\t,\s\}$.
\end{de}

In particular $\t,1$ are ultradifferentiable operators of class $*$ where $*=\{p!^{\t}\}$ or $(p!^{\t})$ in Komastu's notation, cf. \cite{KomatsuNotes}.

\begin{te}
\label{TeoremaZatvorenostUltraDfOP}
Let there be given $\t > 0 $, $\s >1 $ and let
$P(x,\partial)$ be an ultradifferentiable operator of class $(\t,\s)$ (resp. $\{\t,\s\}$).
Then ${\E_{(\infty,\s)}}(U)$ (resp. ${\E_{\{\infty,\s\}}}(U)$) is closed under the action of
$P(x,\partial)$. In particular,
\be
P(x,\partial):\quad {\E}_{\t, \s}(U) \longrightarrow {\E}_{\t
2^{\s-1}, \s}(U)\,,
\ee
is a continuous linear map.
\end{te}

\begin{proof}
Let $a_{\alpha},\phi\in {\E}_{\t, \s,h}(K)$, $\alpha\in \N^d$, $h>0$.
By (\ref{Operatortausigma}) we have
$$
|\partial^{\beta}(a_{\alpha}(x)\partial^{\alpha}\phi(x))|
\leq
\suml_{\gamma\leq \beta}{\beta\choose \gamma}|\partial^{\beta-\gamma}a_{\alpha}(x)||\partial^{\alpha+\gamma}\phi(x)|
$$
$$
\leq A ||\phi||_{{\E}_{\t, \s,h}(K)} \suml_{\gamma\leq \beta}{\beta\choose \gamma}  h^{{|\beta-\gamma|}^{\s}}(|\beta-\gamma|)^{{\t |\beta-\gamma|}^{\s}}
\frac{L^{|\alpha|^{\s}}  h^{{|\alpha+\gamma|}^{\s}}}{|\alpha|^{\t 2^{\s-1}{|\alpha|}^{\s}}}(|\alpha+\gamma|)^{{\t |\alpha+\gamma|}^{\s}}
$$
$$
\leq A||\phi||_{{\E}_{\t, \s,h}(K)}\frac{L^{|\alpha|^{\s}}}{|\alpha|^{\t 2^{\s-1}{|\alpha|}^{\s}}}
(|\alpha+\beta|)^{{\t |\alpha+\beta|}^{\s}}
\suml_{\gamma\leq \beta}{\beta\choose \gamma} h^{{|\beta-\gamma|}^{\s}+{|\alpha+\gamma|}^{\s}}
$$
\begin{equation}
\label{ultradif1}
\leq A ||\phi||_{{\E}_{\t, \s,h}(K)} (C L)^{|\alpha|^{\s}}C^{|\beta|^{\s}}|\beta|^{\t 2^{\s-1}|\beta|^{\s}} C_{h,\beta},
\end{equation}
where we have used the fact that  $M_p^{\t,\s}$ satisfies
$(M.1)'$ and $\widetilde{(M.2)}$,
and put
$\dss C_{h,\beta}=\suml_{\gamma\leq \beta}{\beta\choose \gamma} h^{{|\beta-\gamma|}^{\s}+{|\alpha+\gamma|}^{\s}}$.
Since
$$
\frac{1}{2^{\s-1}} (|\alpha|^{\s}+|\beta|^{\s}) \leq
{|\beta-\gamma|}^{\s}+{|\alpha+\gamma|}^{\s}
\leq
2^{\s-1}(|\alpha|^{\s}+|\beta|^{\s}),
\;\;\; \gamma\leq \beta,
$$
we have
$$C_{h,\beta}\leq 2^{|\beta|}h^{\frac{1}{2^{\s-1}}|\alpha|^{\s}}h^{\frac{1}{2^{\s-1}}|\beta|^{\s}}, \quad 0<h<1,$$
and
$$C_{h,\beta}\leq 2^{|\beta|}h^{2^{\s-1}|\alpha|}h^{2^{\s-1}|\beta|}, \quad h\geq 1.
$$
Put $c_h=\max\{h^{\frac{1}{2^{\s-1}}},h^{2^{\s-1}}\}$. Then
(\ref{ultradif1}) implies
$$
|\partial^{\beta}(a_{\alpha}(x)\partial^{\alpha}\phi(x))|\leq B ||\phi||_{{\E}_{\t, \s,h}(K)}
(c_h C L)^{|\alpha|^{\s}}(2c_hC)^{|\beta|^{\s}}|\beta|^{\t 2^{\s-1}|\beta|^{\s}}.
$$
Choosing $h>0$ (resp. $L>0$) such that $L C c_h< 1/2$,
after summation with respect to $\alpha\in\N^d$,
and by taking suprema with respect to $\beta\in \N^d$ and $x\in K$ it follows
that there exist $C'>0$ such that
$$||P(x,\partial)\phi||_{{\E}_{\t2^{\s-1}, \s,2 C c_h}(K)}\leq C' ||\phi||_{{\E}_{\t, \s,h}(K)}\,$$
which completes the proof.
\end{proof}

It immediately follows that
${\E_{(\infty,\s)}(U)}$ (resp. ${\E_{\{\infty,\s\}}}(U)$) is
closed under the action of
$P(\partial)=\suml_{|\alpha|=0}^{\infty}a_{\alpha}{\partial}^{\alpha}$,
where
$ \displaystyle |a_{\alpha}|\leq A \frac{L^{|\alpha|^{\s}}}{|\alpha|^{\t 2^{\s-1}{|\alpha|}^{\s}}}, $
for some $L>0$ and $A>0$  (resp. every $L>0$ there exists $A>0$).

\section{Microlocal analysis with respect to ${\E_{\t,\s}}(U)$}\label{sec-2}

In this section we define wave front sets which detect singularities that are
"stronger" then classical $C^{\infty}$ singularities and "weaker" then Gevrey type singularities.

In the study of regularity properties (as opposed to the singularity properties)
of a function (or distribution) $u$
we are interested in points
$(x_0,\xi_0)$ in which the decrease of $|\widehat{\phi_N u}(\xi)|$
($\{\phi_N \}_{N\in \N}$ is appropriate sequence of cut-off functions, $\phi_N=1$, $N\in \N$, in a neighborhood if $x_0$)
is faster than $|\xi|^{-N}$ for any $N\in \N,$
and, at the same time, slower than $e^{-|\xi|^{1/t}}$ for any $t>1$,
when $|\xi| \to \infty$ and belongs to an open cone which contains $\xi_0$.
In other words, $u$ is micro-locally more regular than being $C^{\infty}-$regular, but less than being Gevrey
regular.

As a motivation for the definition of wave-front sets in the context of the above mentioned regularity
we observe  the following conditions.

\begin{lema} \label{LemaUslovi}
Let $ t\geq 1 $ and let $\{u_N\}_{N\in \N}$ be a sequence of functions in $C^{\infty}_K$,
such that some of the following conditions hold for every
$ N\in { \N},$ and $\xi\in \Rd\backslash\{0\}$:
\begin{equation}
\label{uslov1}
|\widehat u_N(\xi)|\leq A\, \frac{h^{N^t}\lfloor N^t \rfloor!}{|\xi|^{\lfloor N^t \rfloor}},
\end{equation}
\begin{equation}
\label{uslov2}
|\widehat u_N(\xi)|\leq A\, \frac{h^{N}N!^t }{|\xi|^{ N }},
\end{equation}
\begin{equation}
\label{uslov3}
|\widehat u_N(\xi)|\leq A\, \frac{h^{N} N!^{1/t} }{|\xi|^{\lfloor N^{1/t} \rfloor}},
\end{equation}
for some (different) constants $A,h>0$. Then
$(\ref{uslov1})\Rightarrow (\ref{uslov2})\Rightarrow (\ref{uslov3})\,.$
\end{lema}

As mentioned in the introduction, (\ref{uslov2}) is related to the Gevrey wave front $\WF_t$,  $t>1$,
and if $t=1$ then \eqref{uslov1} - \eqref{uslov3}
are related to the analytic wave front set $\WF_A$.

The proof of Lemma \ref{LemaUslovi}
is just an application of the procedure which we call {\em enumeration}
and which consists of  a change of variables in indices which "speeds up" or "slows down"
the decay estimates of single members of the corresponding sequences,
while preserving their asymptotic behavior when $N \rightarrow \infty$. In other words,
although estimates for terms of a sequence before and after
enumeration are different, the  asymptotic behavior of the whole sequence remains unchanged.

In other words, the conditions of the form (\ref{uslov1}), \eqref{uslov2} or \eqref{uslov3}
are equivalent if one is obtained from another one after replacing $N$ with positive, increasing sequence $a_N$
such that $a_N\to \infty$, $N\to \infty$. We call this procedure  {\em enumeration},
and write $N\to a_N$ and $u_N$ instead of $u_{ a_N }$.

Now, for the proof of Lemma \ref{LemaUslovi}, it is enough to note that
after  enumeration $N\to N^{1/t}$, $ t>1$, (\ref{uslov1}) is equivalent to  local analyticity,
and it immediately follows that
$(\ref{uslov1})\Rightarrow (\ref{uslov2})$. Next, after  enumeration
$N\to N^{t}$, $ t>1$,  \eqref{uslov3} is equivalent to
$$
|\widehat u_N(\xi)| \leq A\, \frac{h^{N^t} \lfloor N^t \rfloor!^{1/t} }{|\xi|^N},
\quad N\in { \N},\xi\in \Rd\backslash\{0\},
$$
and  $(\ref{uslov2})\Rightarrow (\ref{uslov3})$ follows from \eqref{niz1}.

Next we introduce new regularity condition and discuss its relation to the
conditions of Lemma \ref{LemaUslovi}.

Let $ \tau >0 $, $ \sigma \geq 1 $ and let $\{u_N\}_{N\in \N}$ be a sequence of compactly supported smooth functions
such that
\begin{equation}
\label{uslovTauSigma}
|\widehat u_N(\xi)|\leq A\,\frac{h^{N} N!^{1/\s} }{|\xi|^{\lfloor (N/\t)^{1/\s} \rfloor}},
\quad N\in { \N},\,\xi\in \Rd\backslash\{0\}
\end{equation}
for some constants $A,h>0$.
Note that from after enumeration $N\to \t N^{\s}$, \eqref{niz1} implies that
(\ref{uslovTauSigma}) is equivalent to
$$
|\widehat u_{N}(\xi)|\leq A\, \frac{h^{N^{\s}} N^{\t N^{\s}} }{|\xi|^{N}}, \quad N\in { \N},\,\xi\in \Rd\backslash\{0\},
$$
and from $N!^{\s} \leq C N^{\t N^{\s}}$ it follows that
$(\ref{uslov2})\Rightarrow (\ref{uslovTauSigma})$.
Note that
$ (\ref{uslov3})  \Leftrightarrow (\ref{uslovTauSigma}) $
when $ \tau = 1 $, while
$ (\ref{uslovTauSigma}) \Rightarrow (\ref{uslov3})$ when $\tau \in (0,1)$.

We conclude that $(\ref{uslovTauSigma})$ describes regularity weaker than
$(\ref{uslov2})$ and stronger than $  (\ref{uslov3})$.

After applying Stirling's formula and enumeration $N\to N/\t$ to
\begin{equation}
\label{uslov3'}
|\widehat u_N(\xi)|\leq A\, \frac{h^{N} N!^{\t/\s}}{|\xi|^{\lfloor N^{1/\s} \rfloor}},
\quad N\in { \N},\,\xi\in \Rd\backslash\{0\}
\end{equation}
where $A,h >0$, we obtain
$$
|\widehat u_N(\xi)|\leq A\, \frac{h^{N/\t} (N/\t)^{\frac{\t}{\s} (N/\t)} }{|\xi|^{\lfloor (N/\t)^{1/\s} \rfloor}}
\leq B\, \frac{k^{N} N!^{1/\s}}{|\xi|^{\lfloor (N/\t)^{1/\s} \rfloor}},
\quad N\in { \N},\,\xi\in \Rd\backslash\{0\},
$$
for some $A,B,h,k>0$, so that
\eqref{uslov3'} is equivalent to \eqref{uslovTauSigma}.

This discussion motivates the use of  \eqref{uslovTauSigma} (or \eqref{uslov3'})
in the definition of a new type of wave front sets of distributions, see Definition \ref{Wf_t_s}.

\subsection{$\t,\s$-admissible sequences and local regularity of Gevrey ultradistributions}\label{NewTypeSekcija}
An essential tool in our study is the use of carefully chosen
sequences of cut-off functions, defined as follows.

\begin{de}
\label{definicijaNiza}
Let $\t>0$, $\s>1$, and $\Omega\subseteq K\subset\subset U$,
such that $\overline{\Omega} $ is strictly contained in $ K$.
A sequence $\{\chi_N\}_{N\in \N}$ of functions in $C^{\infty}_K$
is said to be {\em  $\t,\s$-admissible  with respect to $K$} if
\begin{itemize}
\item[a)] $\chi_N=1$ in a neighborhood of $\Omega$, for every $N\in \N$,
\item[b)] there exists a positive sequence $C_{\beta}$ such that
\begin{equation}
\label{ocenaNiz}
\sup_{x\in K}
|D^{\alpha+\beta}\chi_{N}(x)|\leq C_{\beta}^{|\alpha|+1}
\lfloor N^{1/\s}\rfloor ^{|\alpha|},
\quad |\alpha|\leq
\lfloor (N/\t)^{1/\s}  \rfloor,
\end{equation} for every $N \in \N$ and  $\beta\in {\N}^d$.
\end{itemize}
\end{de}

When  $\t=\s=1$ we recover the sequence
used by H\"ormander in the study of the analytic behavior of distributions.

Although the following Lemma is
a consequence of \cite[Theorems 1.3.5 and 1.4.2]{HermanderKnjiga}
we give its proof since it contains an important construction which will be used in the sequel.

\begin{lema}
\label{ogranicenostNiza}
Let there be given $r>0$, $\t>0$, $\s>1$ and $x_0\in \Rd$.
There exists $\t,\s$-admissible sequence $\{\chi_N\}_{N\in \N}$ with respect to
$\overline{B_{2r}(x_0)}$ such that $\chi_N=1$ on $B_{r}(x_0)$, for every $N\in \N$.
\end{lema}
\begin{proof}
Fix $r>0$. Let $\dss d_{k}=\frac{r}{4\lfloor (N/\t)^{1/\s}  \rfloor}$,
$\dss k\leq \lfloor (N/\t)^{1/\s}  \rfloor$, $N\in \N$. Note that
$$
\suml_{k=1}^{\lfloor (N/\t)^{1/\s}\rfloor} d_k = \frac{r}{4}<\frac{r}{2},
$$ for every $N\in \N$.

Since the infimum of distances between points in $\overline{B_{5r/4}(x_0)}$ and $\Rd\backslash B_{7r/4}(x_0)$
is $r/2$, from \cite[Theorem 1.4.2]{HermanderKnjiga} it follows that for every $N\in \N$
there exists a smooth function $\widetilde{\chi_N}$ such that
$\supp \widetilde{\chi_N} \subseteq B_{7r/4}(x_0)$, $\widetilde{\chi}_N=1$ on $B_{5r/4}(x_0)$,
and
\begin{equation}
\label{defChiNBezalfa}
\sup_{x\in K}
|D^{\alpha}\widetilde{\chi_{N}}(x)|\leq A^{|\alpha|}\prod_{k=1}^{|\alpha|}d_k=
A^{|\alpha|} \lfloor (N/\t)^{1/\s}\rfloor ^{|\alpha|}\leq
C^{|\alpha|} \lfloor N^{1/\s}\rfloor ^{|\alpha|} ,
\end{equation}
for $|\alpha|\leq \lfloor (N/\t)^{1/\s}\rfloor$, $ N \in \N$, where $C>0$ depends on $\t$ and $\s$.

Next, let $\theta $ be a non-negative function such that $\theta\in C_0^{\infty} (B_{r/4}(x_0))$ and
$ \int \theta (x) dx = 1$.
Then
$\chi_N=\theta*\widetilde{\chi_N}$ clearly satisfies (\ref{ocenaNiz}) for every $N\in \N$,
if we let $\beta$ derivatives act on $\theta$ and $\alpha$ derivatives act on $\widetilde{\chi_N}$.
Hence $\{\chi_N\}_{N\in \N}$ is
a $\t, \s$-admissible sequence with respect to $\overline{B_{2r}(x_0)}$ and the lemma is proved.
\end{proof}

\begin{rem}
\label{PaleyRemark}
Note that if $\alpha=0$ in \eqref{ocenaNiz}, then  $\{\chi_N\}_{N\in \N}$ is a bounded
sequence in $C^{\infty}(U)$. Moreover, by standard calculations we have that
\be
\label{chiNPaley}
|\widehat \chi_N (\xi)|\leq A_{\beta}^{|\alpha|+1}\lfloor N^{1/\s}\rfloor ^{|\alpha|}\langle\xi\rangle^{-|\alpha|-|\beta|},\quad
|\alpha|\leq
\lfloor (N/\t)^{1/\s}  \rfloor,
\ee
for every $N \in \N, \xi\in \Rd$, where $\langle\xi\rangle=(1+|\xi|^2)^{1/2}$.
Therefore, if $u\in \D'(U)$, the sequence $\{\chi_N u\}_{N\in \N}$ is bounded in $\E'(U)$.
\end{rem}

Local regularity in $ \E_{\{\t,\s\}}(U)$ is in fact determined by (\ref{uslov3'}) as follows.

\begin{prop}
\label{dovoljanUslov}
Let $u \in \D'(U)$, and let $\{u_N\}_{N\in \N}$ be a bounded sequence in $\E'(U)$,
$u_N=u$ on $\Omega$ and  such that \eqref{uslov3'}
holds for $\t>0$ and $\s >1$. Then $u\in \E_{\{\t,\s\}}(\Omega)$.
\end{prop}

We omit the proof since it uses
standard arguments based on the Paley-Wiener theorem, the Fourier inversion
formula and suitable decomposition of the domain of integration in combination with
the property $\widetilde{(M.2)'}$. We refer to \cite{HermanderKnjiga} and
\cite{Rodino} for details.

For the opposite direction, if $u\in \E_{\{\t,\s\}}(\Omega)$ then
we use $ \t^{\s/(\s -1)}, \s $-admissible sequences instead.

\begin{prop}
\label{potrebanUslov}
Let $\Omega\subseteq K\subset\subset U$, $\overline{\Omega}$ strictly contained in $ K$,
$u\in \D'(U)$, and let $\{\chi_N\}_{N\in \N}$ be the ${\tilde\t},\s$-admissible sequence with respect to $K$, where ${\tilde\t}= \t^{\s/(\s -1)}$, $\t>0 $, $\s>1$ .
If $u\in \E_{\{\t,\s\}}(\Omega)$,
then $\{\chi_N u\}_{N\in \N}$ is bounded in $\E'(U)$,
${\chi}_N u=u$ on $\Omega$, and
\begin{equation}
\label{uslov2ro}
|\widehat{\chi_N u}(\xi)|\leq
A \frac{h^{N} N!^{{\tilde\t}^{-1/\s}/{\s}} }{|\xi|^{\lfloor ({N/{\tilde\t}})^{1/\s} \rfloor}}, \quad N\in { \N},\,\xi\in \Rd\backslash\{0\}.
\end{equation} That is, after enumeration $N\to {\tilde \t}N$, $\{\chi_N u\}_{N\in \N}$
satisfies \eqref{uslov3'}
for some $A,h>0$.
\end{prop}

The proof is rather technical and follows the same idea as in
\cite[Proposition 8.4.2]{HermanderKnjiga}. We therefore omit it.

Note that for $\t=\s=1$ Proposition \ref{potrebanUslov} coincides with the necessity part of \cite[Proposition 8.4.2.]{HermanderKnjiga}.

\subsection{Singular support and  $\WF_{\t,\s}$ related to the classes $\E_{\t,\s}$}\label{singularsuport}

In this section we introduce wave front set ${\WF}_{\{\t,\s\}}(u)$,
and prove the corresponding results related to singular support.
We also discuss the wave-front set ${\WF}_{(\t,\s)}(u)$.

\begin{de}
\label{Wf_t_s}
Let $\t>0$ and $\s>1$, $u \in \D'(U)$, and $(x_0,\xi_0)\in U\times\Rd\backslash\{0\}$.
Then $(x_0,\xi_0)\not \in {\WF}_{\{\t,\s\}}(u)$ (resp. ${\WF}_{(\t,\s)}(u)$) if there exists
open neighborhood $\Omega \subset U$ of $x_0$,
a conic neighborhood $\Gamma$ of $\xi_0$,
and a bounded sequence
$\{u_N\}_{N\in \N}$ in $\E'(U)$ such that $u_N=u$ on $\Omega$ and
\eqref{uslov3'} holds
for some constants $A,h>0$ (resp. for every $h>0$ there exists $A>0$).
\end{de}

\begin{rem}
It follows immediately from the definition that ${\WF}_{\{\t,\s\}}(u)$, $u\in \D'(U)$,
is closed subset of $U\times\Rd\backslash\{0\}$. Note that for $\t>0$ and $\s>1$
$$
{\WF}_{\{\t,\s\}}(u)\subseteq {\WF}_{\s}(u) \subseteq {\WF}_{\{1,1\}}(u)={\WF}_A (u),
$$
where $ {\WF}_{\s}(u)  $ is the Gevrey wave-front set.
Moreover, when $ 0<\t<1$ and $\s = 1$ we have $ {\WF}_A (u) \subseteq {\WF}_{\{\t,1\}}(u),$ and $\WF_{\{\t,\s\}}(u)\not=\WF_L(u)$ for any choice of $\t>0$, $\s>1$, where $\WF_L(u)$ is given in the introduction.

Since Proposition \ref{potrebanUslov} does not hold when $0<\t<1$ and $\s=1$, we are not able to prove the usual relation between ${\WF}_{\{\t,1\}}(u)$
and the singular support of $u$, see Theorem \ref{SingularSupportTheorema}.
This suggests that the singularities related to $WF_{\{\t,1\}}$
should be studied by a different approach (see \cite{Pilipovic-Toft}).
\end{rem}

The singular support of a  distribution with respect to classes $\E_{\{\t,\s\}}$ can be defined in a usual manner.

\begin{de}
\label{DefinicijaSingSup}
Let $\t>0$, $\s>1$,  $u\in\D'(U)$  and $x_0\in U$.
Then $x_0\not \in \sing_{\{\tau,\s\}}(u)$ if and only if
there exists a neighborhood $\Omega$ of $x_0$ such that $u\in \E_{\{\t,\s\}}(\Omega)$.
\end{de}

The following lemma is an essential result on microlocal regularity,
which will be used in the proof of Theorem \ref{SingularSupportTheorema}.

\begin{lema}
\label{Singsuplema}
Let $\t>0$, $\s>1$,  $u\in\D'(U)$, $K\subset\subset U$,
and let $\{\chi_N\}_{N\in \N}$ be a ${\tilde \t}, \s$-admissible sequence with respect to $K$ with ${\tilde\t}= \t^{\s/(\s -1)}$.
Then $\{\chi_N u\}_{N\in \N}$ is a bounded sequence in $\E'(U)$,
and if ${\WF}_{\{\t,\s\}}(u) \cap (K\times F)=\emptyset$,
where $F$ is a closed cone, then there exist $A,h>0$ such that
\be
\label{SingsuplemaUslov}
|\widehat{\chi_N u}(\xi)|\leq
A \frac{h^{N} N!^{{\tilde\t}^{-1/\s}/{\s}} }{|\xi|^{\lfloor ({N/{\tilde\t}})^{1/\s} \rfloor}},
\quad N\in {\N}\,,\xi\in F\,.
\ee
\end{lema}

The main ingredient of the proof is $\tilde\t, \s$-admissibility of
$\{\chi_N\}_{N\in \N}$ and carefully chosen enumeration applied to \eqref{uslov3'}. Apart from this technical
conditions we may use
the same idea as for the proof of \cite[Lemma 8.4.4.]{HermanderKnjiga} and therefore omit the details.

As a consequence of Propositions \ref{dovoljanUslov}, \ref{potrebanUslov}, and Lemma \ref{Singsuplema}
we obtain the following Theorem.

\begin{te}
\label{SingularSupportTheorema}
Let $\t>0$, $\s>1$, $u\in \D'(U)$, and let
$\pi_1:\Rd\times\Rd\backslash \{0\}\to \Rd$ be the standard projection given with $\pi_1(x,\xi)=x$. Then
$$ 
 \sing_{\{\t,\s\}}(u)=\pi_1(\WF_{\{\t,\s\}}(u)).
$$
\end{te}

\begin{proof}
Fix $x_0\not \in \pi_1(\WF_{\{\t,\s\}}(u))$ and
let $K$ be its compact neighborhood so that $\WF_{\{\t,\s\}}(u)\cap (K\times \Rd\backslash\{0\})=\emptyset$.
By Lemma \ref{Singsuplema} there exists a bounded sequence $\{u_N\}_{N\in \N}$ in $\E'(U)$
such that $u_N=u$ on some open set $\Omega$ and, after enumeration $N\to {\tilde \t}N$,
\begin{equation}
\label{uslovSingsup}
|\widehat u_N(\xi)|\leq A\, \frac{h^{N} N!^{\t/\s} }{|\xi|^{\lfloor N^{1/\s} \rfloor}}, \quad N\in { \N},\,\xi\in \Rd\backslash\{0\}.
\end{equation}
holds for some $A,h>0$. From Proposition \ref{dovoljanUslov} it follows that
$u\in \E_{\{\t,\s\}}(\Omega)$, that is, $x_0\not \in \sing_{\{\t,\s\}}(u)$.

Conversely, if $x_0\not \in \sing_{\{\t,\s\}}(u)$,
then there exist neighborhood $\Omega$ of $x_0$ such that $u\in \E_{\{\t,\s\}}(\Omega)$.
By Proposition \ref{potrebanUslov}, there exists a bounded sequence $\{u_N\}_{N\in \N}$ in $\E'(U)$
such that $u_N=u$ on $\Omega$ and (\ref{uslovSingsup}) holds, which implies the desired equality.
\end{proof}

To conclude the section we discuss
intersections and unions of wave-front sets $\WF_{\t,\s}$, $\t>0$, $\s>1$.
It turns out that, from the microlocal point of view, the regularity related to complements
of these unions and intersections in intimately related to the regularity
properties in the classes given by  (\ref{BorderLinePresek}) and (\ref{BorderLineUnija}).

Let there be given $u\in \D'(U)$. Then we put
\be
\label{WF01}
\WF_{0,1}(u)=\bigcap_{\s>1}\bigcap_{\t>0}\WF_{\t,\s}(u),
\ee
\be
\label{WFinfty1}
\WF_{\infty,1}(u)=\bigcap_{\s>1}\bigcup_{\t>0}\WF_{\t,\s}(u),
\ee
\be
\label{WF0infty}
\WF_{0,\infty}(u)=\bigcup_{\s>1}\bigcap_{\t>0}\WF_{\t,\s}(u),
\ee
\be
\label{WFinftyinfty}
\WF_{\infty,\infty}(u)=\bigcup_{\s>1}\bigcup_{\t>0}\WF_{\t,\s}(u).
\ee

\begin{rem}
\label{RemarkUnijaPresek}
Recall (cf. Proposition \ref{RoumieuBeurling}),
\be
\E_{\{\t,\s\}}(U)\hookrightarrow \E_{(\rho,\s)}(U)\hookrightarrow \E_{\{\rho,\s\}}(U),
\ee
when $0<\t<\rho$ and $\s>1$. Since the inclusions are strict, Definition \ref{Wf_t_s} implies
$$
\WF_{\{\rho,\s\}}(u)\subseteq
\WF_{(\rho,\s)}(u)\subseteq \WF_{\{\t,\s\}}(u)\,, u\in \D'(U).
$$
Moreover, $\dss \bigcap_{\t>0}\WF_{\{\t,\s\}}(u)=\bigcap_{\t>0}\WF_{(\t,\s)}(u)$ and
$\dss \bigcup_{\t>0}\WF_{\{\t,\s\}}(u)=\bigcup_{\t>0}\WF_{(\t,\s)}(u)$. For that reason
it is sufficient to consider intersections and unions of $\WF_{\{\t,\s\}}(u)$ in
\eqref{WF01}-\eqref{WFinftyinfty}.
\end{rem}

First we prove the following technical result.

\begin{lema}
\label{PresekUnijaWF}
Let $u\in \D'(U)$, and $\s_2>\s_1\geq 1$. Then
$$\bigcup_{\t>0}\WF_{\t,\s_2}(u)\subseteq \bigcap_{\t>0}\WF_{\t,\s_1}(u)\,.$$
\end{lema}
\begin{proof}
Let $(x_0,\xi_0)\not \in \bigcap_{\t>0}\WF_{\{\t,\s_1\}}(u)$. Then there exists $\t_0>0$ such that $(x_0,\xi_0)\not \in \WF_{\{\t_0,\s_1\}}(u)$. Hence there exists open conic neighborhood $\Omega\times \Gamma$  of $(x_0,\xi_0)$ and  a bounded sequence $\{u_N\}_{N\in \N}$ in $\E'(U)$ such that $u_N=u$ on $\Omega$ such that, after enumeration $N\to N^{\s_1}$ (see also Lemma \ref{osobineM_p_s}),
\be
\label{OcenaPresekUnija}
|\widehat u_N (\xi)|\leq A \frac{h^{N^{\s_1}}N^{\t_0 N^{\s_1}}}{|\xi|^N},\quad N\in \N, \xi \in \Gamma,
\ee for some constants $A,h>0$.

We need to prove that for every $\t>0$, $(x_0,\xi_0)\not \in \WF_{\{\t,\s_2\}}(u)$. This follows easily from (\ref{OcenaPresekUnija}), noting that (see the proof of the \cite[Proposition 2.1.]{PTT-01}) for every $\t>0$ and $h>0$ there exists $A_1>0$ such that
$$
{h^{N^{\s_1}}N^{\t_0 N^{\s_1}}}\leq A_1 {h^{N^{\s_2}}N^{\t N^{\s_2}}}, \quad N\in \N,
$$
and the Lemma is proved.
\end{proof}

As a consequence of Lemma \ref{PresekUnijaWF} we obtain the following result
which relates our regularity with
$C^{\infty}$ and
$\E_t$-regularity in terms of the corresponding wave-front sets.

\begin{cor} \label{strict inclusions}
Let  $u\in \D'(U)$. Then, in the notation of
\eqref{WF01}-\eqref{WFinftyinfty}, we have
\begin{multline}
\WF(u) \subseteq\WF_{0,1}(u)\subseteq \WF_{\infty,1}(u)\\
\subseteq \WF_{0,\infty}(u)\subseteq \WF_{\infty,\infty}(u)\subseteq \bigcap_{\t>1}\WF_{\t}(u)\,,
\end{multline}
where $\WF$ and $\WF_{\t}$ are the classical and the Gevrey wave-front sets, respectively.
\end{cor}
\begin{proof}
Note that the last inclusion follows from Lemma \ref{PresekUnijaWF} for $\s_2>\s_1=1$
by taking unions and intersections with respect to $\t>1$.
The only nontrivial inclusion is
$ \WF_{\infty,1}(u) \subseteq \WF_{0,\infty}(u)$.
Assume that $(x_0,\xi_0)\not \in \WF_{0,\infty}(u)$, that is,
$\dss(x_0,\xi_0)\not \in \bigcap_{\t>0}\WF_{\t,\s}(u)$, for every $\s>1$.
Fix some $\s=\s_1>1$ and let $\s_2>\s_1$. By Lemma \ref{PresekUnijaWF} it follows that
$(x_0,\xi_0)\not \in\bigcup_{\t>0} \WF_{\t,\s_2}(u)$. Hence
there exists $\s>1$ such that for every $\t>0$
$(x_0,\xi_0)\not \in \WF_{\t,\s}(u)$ and therefore  $(x_0,\xi_0)\not \in \WF_{\infty,1}(u)$.
\end{proof}

To end the section, we relate $\WF_{0,\infty}(u)$ to the regularity
in  $\E_{\infty,1}$, see (\ref{BorderLineUnija}).
Let $\sing_{\infty, 1}(u)$ denote
the singular support of $u \in  \D'(U)$
related to the classe $\E_{\infty,1}$ (as appropriate union and intersection of the
corresponding singular supports in $ \E_{\t,\s}(U)$)
Recall that, for every $\s>1$,
the space $\E_{\infty,\s}$ is closed under the action of ultradifferentiable operators of the class
$\t,\s$ (see Subsection\ref{svojstva}, Theorem \ref{TeoremaZatvorenostUltraDfOP}).
Then, arguing in the similar way as in the proof of Theorem \ref{SingularSupportTheorema}
one can prove that
\be
\label{singSupUnijaPresek}
\pi_1(\WF_{0,\infty}(u))=\sing_{\infty,1}(u).
\ee

\section{Proof of Theorem \ref{glavnateorema}} \label{sec-3}

Note that Corollary \ref{FundamentalEstimateCor} follows directly from
Theorem \ref{glavnateorema} and Remark \ref{RemarkUnijaPresek}.
The first embedding in (\ref{FundamentalEstimate}) immediately
follows form the next Lemma.

\begin{lema}
\label{zatvorenostWfizvodi}
Let $u\in \D'(U)$, $\t>0, \s>1$. Then
$$
\WF_{\{\t,\s\}}(\partial_j u)\subseteq\WF_{\{\t,\s\}}(u),
$$
for all $1\leq j\leq d$.
\end{lema}

\begin{proof}
Let $(x_0,\xi_0)\not \in \WF_{\t,\s}(u)$. Then
there exists a conical neighborhood $\Omega\times \Gamma$ of $(x_0,\xi_0)$
and a bounded sequence $\{u_N\}$ in $u\in \dss\E'(U)$ such that $u_N=u$ on $\Omega$,
and such that after the enumeration $N\to N^{\s}$ we obtain
\begin{equation}
|\widehat u_{N}(\xi)|\leq A\, \frac{h^{N^{\s}} N^{\t N^{\s}} }{|\xi|^{N}}, \quad N\in { \N},\,\xi\in \Gamma,
\end{equation}
for some $A,h>0$.
Then, for $x_0\in \Omega$,
\be
|\widehat{\partial_j u}_{N+1}(\xi)|
\leq A |\xi|\, \frac{h^{N} (N+1)^{\t (N+1)^{\s}} }{|\xi|^{N+1}}\leq  A_1\,
\frac{h_1^{N} N^{\t N^{\s}} }{|\xi|^{N}} ,
\ee
$N\in { \N},$ $\xi\in \Gamma,$ $j \in \{ 1,\dots, d\}$,
$\widetilde{(M.2)}'$ is used for the second inequality, and the inclusion follows.
\end{proof}

Therefore it remains to prove that
$$
\WF_{\{2^{\s-1}\t,\s\}}(u)
\subseteq
\WF_{\{\t,\s\}}(P(D)u) \cup {\rm Char}(P(D)).
$$

The following inequality, which holds for $\t>0$, $\s>1$ and for some $C>0$,
will be frequently used:
\be
\label{SimpleInequality}
{\lfloor N^{1/\s} \rfloor}^{\lfloor (N/\t)^{1/\s} \rfloor}
\leq{ N }^{N \t^{-1/\s}/\s}\leq C^{N}{ N! }^{\t^{-1/\s}/\s}.
\ee

Assume that $(x_0,\xi_0)\not \in\WF_{\{\t,\s \}}(P(D)u)\cup {\rm Char}(P(D))$.
Then there exists a compact set $K$ containing $x_0$
and a closed cone $ \Gamma$ containing $ \xi_0 $ such that
$P_m(x,\xi)\not=0$ when $(x,\xi) \in  K \times \Gamma$ and
$(K\times \Gamma)\cap \WF_{\{\t,\s \}}(P(D)u)=\emptyset$.

Let $\tilde \t = \t^{\frac{\s}{\s -1}} $ and let
$\{\chi_N\}_{N\in \N}$, be a $\tilde \t, \s$-admissible sequence  with respect to $K$.

Put $u_N=\chi_{  2^{\s}N  }u$, $N \in \N$, so that
$$
\widehat u_N(\xi) =  \int u(x) \chi_{ 2^{\s}N  }(x)e^{-i x\xi} dx, \;\;\; \xi\in \Rd,\; N \in \N.
$$

The easy part of the proof is the
estimate of $ |\widehat u_N(\xi) | $,  $N \in \N$, for "small" values of $\xi \in \Gamma$,
that is when $|\xi| \leq \lf N^{1/\s}  \rf $.
In fact, since $\{u_N\}_{N\in \N}$ is bounded in $\E'(U)$,
Paley-Wiener theorems (see \cite{KomatsuNotes}), and the fact that $e^{-ix\cdot\xi}\in C^{\infty}(\Rd_x)$,
for every $\xi\in \Rd$, implies that
$ \displaystyle |\widehat u_N(\xi)|=|\langle u_N,e^{-i\cdot\xi}\rangle|\leq C\langle\xi\rangle^M,$
for some $C,M>0$ independent of $N$. Hence, from \eqref{SimpleInequality} we have
$$
|\xi|^{\lf (N/\tilde \t)^{1/\s} \rf}|\widehat{u_N}(\xi)|
\leq \lf N^{1/\s}  \rf^{\lf (N/\tilde \t)^{1/\s} \rf}|\widehat{u_N}(\xi)|
\leq A C^N N^{\frac{\tilde \t^{-1/\s}}{\s}N},
$$
where $A,C>0$ do not depend on $N$. After enumeration $N\to \tilde \t N$ we obtain
$$
|\widehat{u_N}(\xi)|\leq A \frac{C^N N^{\frac{\tilde \t^{1-1/\s}}{\s}N}}{|\xi|^{\lf N^{1/\s} \rf}}
\leq A \frac{h^N N!^{\frac{\t}{\s}}}{|\xi|^{\lf N^{1/\s} \rf}},
$$
which estimates $|\widehat{u_N}(\xi)| $  when $\xi \in \Gamma$, $ |\xi| \leq \lf N^{1/\s}  \rf,$ $ N \in \N.$

It remains to estimate  $ |\widehat u_N(\xi) | $,   when
$\xi \in \Gamma$,  $|\xi| > \lf N^{1/\s}  \rf $ and for $N \in \N$ large enough
(so that $ N \rightarrow \infty $ implies $|\xi| \rightarrow \infty$).

\par

As in the proof of \cite[Theorem 8.6.1]{HermanderKnjiga}, in
Subsection  \ref{subsecrepres} we use  the technique of  approximate solution
(see also \cite[Theorem 1, Section 1.6]{Rauch})
to obtain
\be
\label{fundamentalEquality}
\chi_{2^{\s} N}(x)e^{-ix\cdot\xi} =
  P^T (D) \left (
\frac{e^{-i x\cdot \xi}}{P_m(\xi)} w_N (x, \xi)
\right)+  e_N (x, \xi)e^{-ix\xi}\,
\ee
$x\in K,$ $\xi \in \Gamma$,  $|\xi| > \lf N^{1/\s}  \rf $,
that is, the following representation holds:
\begin{multline}
\label{ocenitiJedn}
\widehat u_N(\xi)  =  \int u(x) e_N (x, \xi) e^{-i x\xi} dx +
 \int u(x) P^T (D) \left (
\frac{e^{-i x\cdot \xi} w_N (x, \xi)}{P_m(\xi)}
\right)  dx \\
=  \int u(x) e_N (x, \xi) e^{-i x\xi} dx +
 \int P (D)  u(x) \left (
\frac{e^{-i x\cdot \xi} w_N (x, \xi)}{P_m(\xi)}
\right)  dx,
\end{multline}
where
\begin{equation}
\label{ParcijalnaSuma}
w_N(x,\xi) =\sum_{\mathfrak{S} =0} ^{\lf (\frac{N}{\tilde \t})^{\frac{1}{\s}} \rf-m}
{|a|\choose a_1,a_2\dots a_m} (R_1^{a_1}R_2^{a_2}\dots R_m^{a_m}
\chi_{2^{\s} N})(x,\xi),
\end{equation}
\begin{equation}
\label{errorterm}
e_N(x,\xi)= \sum_{k=1}^m \sum_{\mathfrak{S}  = \lf (\frac{N}{\tilde \t})^{\frac{1}{\s}}\rf-m +1}
^{\lf (\frac{N}{\tilde \t})^{\frac{1}{\s}} \rf-m+k}
{|a|\choose a_1,...,a_m}(R_1^{a_1}...R_m^{a_m}
\chi_{2^{\s} N})(x,\xi),
\end{equation}
$ x\in K,$ $ \xi\in \Gamma, $ $|\xi| > \lf N^{1/\s} \rf$,
and we put $ \mathfrak{S} =  a_1+2a_2+ \dots +m a_m$.

The derivation of \eqref{ocenitiJedn}
and the calculation of $ w_N(x,\xi)$ and $e_N(x,\xi)$ is
done in Subsections \ref{subsecrepres} and \ref{last}, so we continue with the estimation of the first term in
\eqref{ocenitiJedn}.

Estimated number of terms in $e_N(x,\xi)$ given in Subsection \ref{subsecrepres},
and the estimates of $ D^\beta (R_1^{a_1}...R_m^{a_m}\chi_{ 2^{\s}N  })$
given by (\ref{OcenaProizvodaOp}) (Subsection \ref{subsecderivatives})
imply
\begin{eqnarray}
\label{OcenaPrvogSab}
|{\langle u(x), e_{N}(x,\xi)e^{-i x\cdot\xi}\rangle}|&\leq & A \sum_{|\alpha|\leq M}|D_x^{\alpha}(e_{N}(x,\xi)e^{-ix\xi})|\nonumber\\
&\leq& A \sum_{|\alpha|\leq M}\sum_{\beta\leq \alpha}{\alpha\choose \beta}|D_x^{\alpha-\beta}e^{-ix\xi}||D_x^{\beta}e_N(x,\xi)|\nonumber\\
&\leq& A |\xi|^M |\xi|^{-\lf 2^{\frac{1-\s}{\s}}(N/\tilde \t)^{1/\s} \rf-M}
C^N N!^{\frac{\tilde \t^{-1/\s}}{\s}}\nonumber\\
&=& A\frac{C^N N!^{\frac{\tilde \t^{-1/\s}}{\s}}}{|\xi|^{\lf 2^{\frac{1-\s}{\s}}(N/\tilde \t)^{1/\s} \rf}},
x\in K, \xi \in \Gamma,
\end{eqnarray}
for suitable constants $A,C>0$ and $|\xi| $ large enough.
After  enumeration $N\to \tilde \t2^{\s-1}N$, (\ref{OcenaPrvogSab}) is equivalent to
$$
|{\langle u(x), e_{N}(x,\xi)e^{-ix\cdot\xi}\rangle}|
\leq A\frac{C^N N!^{\frac{\t 2^{\s-1}}{\s}}}{|\xi|^{\lf N^{1/\s} \rf}},\quad x\in K, \xi \in \Gamma,$$
which estimates the first term on the righthand side of (\ref{ocenitiJedn}).
In fact, we will use a slightly weaker estimate which is obtained from  (\ref{OcenaPrvogSab}) after enumeration
\be \label{anotherenumeration}
N \to N+\lceil{\tilde \t} 2^{\s-1}(M+d+1)^{\s}\rceil.
\ee

It remains to  estimate the  second term on the righthand side of
(\ref{ocenitiJedn}) for $|\xi|>\lf N^{1/\s} \rf$.
This is the hardest part of the proof.
By the Lemma \ref{Singsuplema} there exists a bounded sequence $\{f_N\}_{N\in \N}$
in $\E'(U)$ such that $f_N=f = P(D)u$ in a neighborhood of $K$ and
there exists a cone $V$ such that  $\overline \Gamma \subset V$ and
\be
\label{ocenaZaf}
|{\mathcal F} ({f}_N)(\eta)|\leq A \frac{h^N N!^{\frac{\tilde \t^{-1/\s}}{\s}}}{|\eta|^{\lf (N/\tilde \t)^{1/\s} \rf}},\quad
\eta\in V.
\ee

Since $\{\chi_{ 2^{\s}N  }(x)\}_{N\in \N}$ is bounded in $C_0^{\infty}(U)$,
by the Paley-Wiener theorem (see also Remark \ref{PaleyRemark}) it follows that for every ${\tilde M}>0$ there exists $C>0$ which does not depend on $N$
so that $ |\widehat\chi_{ 2^{\s}N  }(\eta)|\leq C \langle\eta\rangle^{- {\tilde M}}$, $N\in \N$.
From $\supp\chi_N\subseteq K$, $N\in \N$, it follows that
$$\pi_1(\supp w_N(x,\xi))\subseteq K,\quad N\in \N,$$ and since $f_N=f$ in a neighborhood of $K$,
we have $w_N f=w_N f_{N'}$ in $\D'(U)$, where
we put  $N'=N-\lceil 2^{\s-1}\tilde \t (M+d+1)^{\s} \rceil$.
Therefore (and since $ {\mathcal F} (g_1\cdot g_2) (\xi) = ({\mathcal F} (g_1) *  {\mathcal F} (g_2)) (\xi)$))
\begin{multline*}
\langle f(\cdot) e^{-i\xi\cdot},w_N(\cdot,\xi)/P_m(\xi)\rangle
= \frac{1}{P_m(\xi)}{\mathcal F}_{x\to \xi}(f_{N'}(x)w_N(x,\xi))(\xi) \\
=\frac{1}{P_m(\xi)}\int_{{\bf R}^d}{\mathcal F}(f_{N'})(\xi-\eta){\mathcal F}_{x\to \eta}(w_N(x,\xi))(\eta)\,d\eta
=I_1+I_2,
\end{multline*}
where
\be
I_1=\frac{1}{P_m(\xi)}\int_{|\eta|<\varepsilon |\xi|}{{\mathcal F} (f_{N'})}(\xi-\eta){\mathcal F}_{x\to \eta}({{w_N}}(x,\xi))(\eta,\xi)\,d\eta,\quad
\ee
\be
I_2=\frac{1}{P_m(\xi)}\int_{|\eta|\geq\varepsilon |\xi|}{{\mathcal F} (f_{N'})}(\xi-\eta){\mathcal F}_{x\to \eta}({{w_N}}(x,\xi))(\eta,\xi)\,d\eta,
\ee
and $0< \varepsilon <1$ is chosen so that $\xi-\eta\in V$ when $\xi\in \Gamma,$ $\xi >\lf N^{1/\s} \rf$,
and $|\eta|<\varepsilon |\xi|$.

Since $|\eta|<\varepsilon |\xi|$ implies $|\xi-\eta|\geq (1-\varepsilon)|\xi|$,
by using the computation of ${\mathcal F}_{x\to\eta}(w_N)(\eta,\xi)$ from Subsection \ref{subsecfurije}, we estimate
$I_1$ as follows:
\begin{eqnarray}
|I_1|&\leq& \frac{1}{|P_m(\xi)|} \int_{|\eta|<\varepsilon |\xi|}|{\mathcal F} (f_{N'})(\xi-\eta)|
|{\mathcal F}_{x\to \eta}({w_N})(\eta,\xi)|\,d\eta \nonumber\\
&\leq& \int_{|\eta|<\varepsilon |\xi|}A \frac{h^{N'} N'!^{\t/\s}}{|\xi-\eta|^{\lf {N'}^{1/\s} \rf}}
|{\mathcal F}_{x\to \eta}({w_N})(\eta,\xi)|\,d\eta \nonumber\\
&\leq& A \frac{h^{N'} N'!^{\t/\s}}{((1-\varepsilon)|\xi|)^{\lf {N'}^{1/\s} \rf}}
\int_{|\eta|< \varepsilon |\xi|}|{\mathcal F}_{x\to \eta}({w_N})(\eta,\xi)|\,d\eta\nonumber\\
&\leq& A_1 \frac{h_1^{N'} N'!^{\t/\s}}{|\xi|^{\lf {N'}^{1/\s} \rf}} C^{\lf(N/\t)^{1/\s}\rf}\int_{{\bf R}^d}|\widehat\chi_{2^{\s} N}(\eta)|\,d\eta\nonumber\\
&\leq& A_2 \frac{h_2^{N'} N'!^{\t/\s}}{|\xi|^{\lf {N'}^{1/\s} \rf}},
\;\;\; \xi\in  \Gamma, |\xi| > \lf N^{1/\s} \rf.
\end{eqnarray}

We used the Paley-Wiener theorem for $\{\widehat\chi_{2^{\s}N}\}$ and
trivial inequality $|P_m(\xi)|\geq 1$ when $|\xi|>\lf N^{1/\s} \rf$.

It remains  to estimate $I_2$. Note that $|\eta|\geq \varepsilon |\xi|$
implies $|\xi-\eta|\leq (1+1/\varepsilon)|\eta|$, and by Paley-Wiener type estimates we have $ |{\mathcal F}(f_{N'})(\eta)|\leq C \langle\eta\rangle^M,$ where $C>0$ does not depend on $N'$.
Therefore
\begin{multline}
|I_2|\leq \frac{1}{|P_m(\xi)|} \int_{|\eta|\geq\varepsilon |\xi|}|{\mathcal F} f_{N'}(\xi-\eta)|
|{\mathcal F}_{x\to \eta}({w_N})(\eta,\xi)|\,d\eta \nonumber\\
\leq A \int_{|\eta|\geq \varepsilon |\xi|}\langle\xi-\eta\rangle^M
\langle\eta\rangle^{\lf 2^{\frac{1-\s}{\s}}(N'/{\tilde \t})^{1/\s}\rf+d+1}
\frac{|{\mathcal F}_{x\to \eta}({w_N})(\eta,\xi)|}{\langle\eta\rangle^{\lf 2^{\frac{1-\s}{\s}} (N'/{\tilde \t})^{1/\s}\rf+d+1}}
 \,d\eta \nonumber\\
\leq C^{N+1}
\frac{\sup_{\eta\in \Rd}\langle \eta \rangle^{\lf 2^{\frac{1-\s}{\s}}(N'/{\tilde \t})^{1/\s}\rf+M+d+1}
}{|\xi|^{\lf 2^{\frac{1-\s}{\s}}(N'/{\tilde \t})^{1/\s} \rf}} |{\mathcal F}_{x\to \eta}({w_N}(x,\xi))(\eta,\xi)|,
\end{multline}
when $\xi\in  \Gamma$, $|\xi| > \lf N^{1/\s} \rf$.

To finish the proof, we  show that if $\xi\in  \Gamma$, $|\xi| > \lf N^{1/\s} \rf$
then there exists $h>0$ such that
\be
\label{ZelimodaOcenimo}
\sup_{\eta \in \Rd}\langle \eta \rangle^{\lf 2^{\frac{1-\s}{\s}}(N'/{\tilde \t})^{1/\s}\rf+M+d+1}
|{\mathcal F}_{x\to \eta}({w_N})(\eta,\xi)|\leq h^{N+1} N!^{1/\s}.
\ee
Since $N'=N-\lceil 2^{\s-1}\tilde \t (M+d+1)^{\s} \rceil $,
it follows that
\be
\label{NejednakostSaN1}
(N/\tilde \t)^{1/\s}
=
\Big(\frac{N'+\lceil 2^{\s-1}\tilde \t (M+d+1)^{\s} \rceil}{\t}\Big)^{1/\s}
\geq 2^{\frac{1-\s}{\s}}(N'/{\tilde \t})^{1/\s}+M+d+1.
\ee
If $\mathfrak{S}\leq \lf (N/{\tilde \t})^{1/\s} \rf-m$,
$|\beta|=\lf (N/ \tilde \t)^{1/\s} \rf$
then
\be
\label{OcenazaPoslednjiInt}
\mathfrak{S}+|\beta|< 2\lf (N/{\tilde \t})^{1/\s} \rf\leq \lf 2 (N/{\tilde \t})^{1/\s} \rf,
\ee
From (\ref{OcenazaPoslednjiInt}), when $x\in K$ and  $\xi\in  \Gamma$, $|\xi| > \lf N^{1/\s} \rf$
it follows that
$$
|D^{\beta}w_N(x,\xi)|\leq
\sum_{\mathfrak{S} = 0} ^{\lf (\frac{N}{{\tilde \t}})^{\frac{1}{\s}} \rf-m}
{|a|\choose a_1,a_2\dots a_m}
\sup_{x\in K}|(D^{\beta}R_1^{a_1}R_2^{a_2}\dots R_m^{a_m}\chi_{2^{\s} N})(x,\xi)|
$$
$$
\leq \sum_{\mathfrak{S} = 0} ^{\lf (\frac{N}{{\tilde \t}})^{\frac{1}{\s}} \rf-m}
{|a|\choose a_1,a_2\dots a_m}
|\xi|^{-\mathfrak{S}} C^{\mathfrak{S}+|\beta|+1} \lf N^{1/\s} \rf^{\mathfrak{S}+|\beta|}\nonumber
$$
$$
\leq \lf N^{1/\s} \rf^{|\beta|} \sum_{\mathfrak{S}=0} ^{\lf (\frac{N}{{\tilde \t}})^{\frac{1}{\s}} \rf-m}
{|a|\choose a_1,a_2\dots a_m} C^{\mathfrak{S}+|\beta|+1}\\
\leq  C'^{\lf(N/\tilde \t)^{1/\s}\rf+1}\lf N^{1/\s} \rf^{|\beta|}.
$$
Since $\pi_1(\supp w_N(x,\xi))\subseteq K$ and $|\beta|=\lf(N/\tilde \t)^{1/\s}\rf$,
we obtain
\be
\label{glupost11}
|\eta|^{\lf(N/{\tilde \t})^{1/\s}\rf}|{\mathcal F}_{x\to \eta}({w_N})(\eta,\xi)|
\leq  C'^{\lf(N/\tilde \t)^{1/\s}\rf+1}\lf N^{1/\s} \rf^{\lf(N/{\tilde \t})^{1/\s}\rf}\leq C''^{N+1} N^{\frac{{\tilde \t}^{-1/\s}}{\s}N}\,,
\ee
where we used the first part of (\ref{SimpleInequality}).
Now (\ref{NejednakostSaN1}) and (\ref{glupost11}) gives
\begin{multline}
\sup_{\eta \in \Rd}\langle \eta \rangle^{\lf 2^{\frac{1-\s}{\s}}(N'/{\tilde \t})^{1/\s}\rf+M+d+1}
|{\mathcal F}_{x\to \eta}({w_N})(\eta,\xi)|\\
\leq \sup_{\eta\in {\bf R}^d}\langle\eta \rangle^{\lf (N/{\tilde \t})^{1/\s}\rf}
|{\mathcal F}_{x\to \eta}({w_N})(\eta,\xi)|\leq C''^{N+1} N^{\frac{{\tilde \t}^{-1/\s}}{\s}N}\,,
\end{multline} and (\ref{ZelimodaOcenimo}) follows.
Therefore
\be
\label{poslednjiUslov1}
|I_2|\leq A \frac{h^{N} N^{\frac{{\tilde \t}^{-1/\s}}{\s}N}}{|\xi|^{\lf 2^{\frac{1-\s}{\s}}(N'/{\tilde \t})^{1/\s} \rf}}\,,
\ee
for suitable constants $A,h>0$. After enumeration given by
\eqref{anotherenumeration},
and using $(M.2)'$ property of the sequence $N^{\frac{{\tilde \t}^{-1/\s}}{\s}N}$, we conclude that (\ref{poslednjiUslov1}) is equivalent to
\be
|I_2|\leq A \frac{h^{N} N!^{\frac{{\tilde \t}^{-1/\s}}{\s}}}{|\xi|^{\lf 2^{\frac{1-\s}{\s}}(N/{\tilde \t})^{1/\s} \rf}}\,,
\ee
for some $A,h>0$.
After enumeration $N\to {\tilde \t} 2^{\s-1}N$ we finally obtain
$$
|\widehat u_N(\xi)|\leq A\frac{h^N N!^{\frac{\t 2^{\s-1}}{\s}}}{|\xi|^{\lf N^{1/\s} \rf}},
$$
for some $A,h>0$, and the proof is finished.

\subsection{Derivation of the representation of $\widehat u_N(\xi)$} \label{subsecrepres}
Formally, we are searching for  $v(x,\xi)$ so that
$$
\widehat u_N(\xi)  =  \int u(x) \chi_{ 2^{\s}N  }(x)e^{-i x\xi} dx
=  \int u(x) P^T (D) v(x,\xi) dx,
$$
$\xi \in \Gamma$,  $|\xi| > \lf N^{1/\s}  \rf $,
where $P^T (D)=\dss\sum_{|\alpha|\leq m} (-1)^{|\alpha|}a_{\alpha} D^{\alpha}$
is the transpose operator of $P(D)$,
and $v(x,\xi)$ is the solution of the equation
\begin{equation}
\label{ResitiJednacinu}
 P^T (D) v(x, \xi)
= \chi_{ 2^{\s}N  }(x)e^{-i x\xi}, \quad x\in K, \xi\in \Gamma, |\xi| > \lf N^{1/\s}  \rf .
\end{equation}
If $v(x,\xi)$ is of the form
$ \displaystyle v (x, \xi) = \frac{e^{-i x\xi}w(x, \xi)}{P_m(\xi)}, $
for some $ w (\cdot,\xi) \in C^{\infty}(K),$ where
$ x\in K,$ $ \xi\in \Gamma,$  $|\xi| > \lf N^{1/\s}  \rf $,
then \eqref{ResitiJednacinu} becomes
\be
\label{JednacinaKojuResavamo}
(I -R(\xi))w(x,\xi)= \chi_{2^{\s} N}(x)\,
\quad x\in K, \xi \in \Gamma, |\xi| > \lf N^{1/\s}  \rf ,
\ee
where $R(\xi)=\sum_{j=1}^m R_j(\xi)$,
$\dss R_j(\xi)= p_j(\xi) \sum_{|\alpha|\leq j}a_{\alpha}D^{\alpha}$, and $p_j(\xi)$
are homogeneous functions of order $-j$.
In fact, formal calculation gives
\begin{multline}
\nonumber
e^{i x\xi} P^T (D)(\frac{w(x,\xi)e^{-i x\xi}}{{P_m(\xi)}})\\[1ex]
= e^{i x\xi} \frac{1}{P_m(\xi)}\sum_{|\alpha|\leq m}\sum_{\beta\leq\alpha} {\alpha \choose \beta}
(-1)^{|\alpha|}a_{\alpha} D^{\alpha-\beta}(e^{-i x\xi}) D^{\beta}w(x,\xi) \\[1ex]
= \sum_{|\alpha|\leq m} \sum_{\beta\leq\alpha} {\alpha \choose \beta}
(-1)^{|\alpha|}a_{\alpha} \Big( \frac{(-\xi)^{\alpha-\beta}}{P_m(\xi)}\Big)D^{\beta}w(x,\xi),
\end{multline}
for $x\in K$ and $\xi\in \Gamma$,  $|\xi| > \lf N^{1/\s}  \rf $.
Since
$\dss \frac{(-\xi)^{\alpha-\beta}}{P_m(\xi)}$ is homogeneous
of order $|\alpha|-|\beta|-m$ with respect to $\xi$,
it follows that (\ref{ResitiJednacinu}) would imply \eqref{JednacinaKojuResavamo}.

Now, successive applications of the operator $R$ in \eqref{JednacinaKojuResavamo} give
$$
 R^{k-1}(\xi)w(x,\xi) - R^k (\xi)w(x,\xi)=
R^{k-1}(\xi)\chi_{2^{\s} N}(x), \;\; x\in K, \xi \in \Gamma,
|\xi| > \lf N^{1/\s}  \rf,
$$
for every $ k \in \{ 1, \dots, N \} $,
so that after summing up those $N$ equalities we obtain
$$ \displaystyle
 w(x,\xi)  - R^N(\xi)w(x,\xi)= \sum_{k=0} ^{N-1 } R^{k}(\xi)\chi_{2^{\s} N}(x),
$$
which gives formal approximate solution
\begin{multline}
w(x,\xi) = \sum_{k=0} ^{\infty} R^{k}\chi_{2^{\s} N}(x, \xi) \\
= \sum_{|a|=0}^{\infty} {|a|\choose a_1,a_2,\dots,a_m}R_1^{a_1}R_2^{a_2}\dots R_m^{a_m}\chi_{2^{\s} N}
(x, \xi). \label{beskonacnaSuma}
\end{multline}
The operators $R_k^{a_k}(\xi)$, $1\leq k\leq m$,
are of order less then or equal to $k a_k$ and homogeneous of order $-k a_k$ with respect to $\xi$.
Since $P(D)$ have constant coefficients,
the operators $R_j$ commute, and we used
the generalized Newton formula, cf. \cite{Rodino}.

We proceed with the following approximation procedure. We consider partial sums
$$
w_N(x,\xi)
=\sum_{\mathfrak{S} = 0} ^{\lf (\frac{N}{\tilde \t})^{\frac{1}{\s}} \rf-m}
{|a|\choose a_1,a_2\dots a_m} (R_1^{a_1}R_2^{a_2}\dots R_m^{a_m}
\chi_{2^{\s} N})(x,\xi),
$$
$\xi \in \Gamma$,  $|\xi| > \lf N^{1/\s}  \rf $,
and $N\in \N$ is large enough, so that
\eqref{JednacinaKojuResavamo} takes
the form \eqref{fundamentalEquality}
and the error term  $e_N$ is given by:
$$
e_N(x,\xi)=
\sum_{k=1}^m \sum_{\mathfrak{S} =\lf (\frac{N}{\tilde \t})^{\frac{1}{\s}}
\rf-m+1} ^{\lf (\frac{N}{\tilde \t})^{\frac{1}{\s}} \rf-m+k}
{|a|\choose a_1,...,a_m}(R_1^{a_1}...R_m^{a_m}
\chi_{2^{\s} N})(x,\xi).
$$
The precise calculation which leads to \eqref{fundamentalEquality} is given in Subsection \ref{last}.
Note that the number of terms in \eqref{errorterm}
is bounded by $  4\cdot 2^{\lf (\frac{N}{\tilde \t})^{\frac{1}{\s}} \rf},$ since
from ${n\choose k}\leq 2^n$, $k\leq n$, $n\in \N$, we obtain
$$
{|a|\choose a_1,a_2,\dots a_m}\leq 2^{|a|}2^{|a|-a_1}\dots 2^{|a|-a_1-\dots-a_{m-2}}\leq 2^{a_1+2 a_2+\dots+ma_m},
$$
and therefore
$$
\sum_{k=1}^m
\sum_{\mathfrak{S} =\lf (\frac{N}{\tilde \t})^{\frac{1}{\s}} \rf-m +1}
^{\lf (\frac{N}{\tilde \t})^{\frac{1}{\s}} \rf-m+k}{|a|\choose a_1,\dots,a_m}
\leq \sum_{k=1}^m
\sum_{\mathfrak{S} =\lf (\frac{N}{\tilde \t})^{\frac{1}{\s}} \rf-m +1}
^{\lf (\frac{N}{\tilde \t})^{\frac{1}{\s}} \rf-m+k}
2^{a_1+2a_2\dots+m a_m}
$$
$$
\leq  2^{\lf (\frac{N}{\tilde \t})^{\frac{1}{\s}} \rf-m+1}\sum_{k=1}^m 2^k \leq 4\cdot 2^{\lf (\frac{N}{\tilde \t})^{\frac{1}{\s}} \rf},
$$
where we put $ \mathfrak{S} = a_1+2 a_2+\dots+ma_m.$

\subsection{The calculation of the error term} \label{last}
For multinomial coefficients
\begin{multline}
{|a|\choose a_1,a_2,\dots a_m}
:={|a|\choose a_1}{|a|-a_1\choose a_2}\dots{|a|-a_1-\dots-a_{m-2}\choose a_{m-1}} \\[1ex]
=\frac{|a|!}{a_1!a_2!\dots a_m!}, \;\;\; |a|=a_1+a_2+\dots+a_m, \; a_k\in \N, \; k\leq m,
\end{multline}
a generalization of Pascal's triangle equality for the binomial formula gives
\be
\label{PascalTriangle}
{|a|\choose a_1,..., a_m}=\sum_{k=1}^m {|a|-1\choose a_1,..., a_k-1,... a_m},\quad |a|\geq 1,
\ee
wherefrom for $|a|\geq 1$, and putting $ \mathfrak{S} = a_1+2 a_2+\dots+ma_m$ we obtain
$$
\sum_{\mathfrak{S}=0} ^{(\lf \frac{N}{\t})^{\frac{1}{\s}} \rf-m}
{|a|\choose a_1,..., a_m} R_1^{a_1}...R_m^{a_m}\chi_{2^{\s} N}
$$
$$
=\sum_{\mathfrak{S}=0} ^{ \lf (\frac{N}{\t})^{\frac{1}{\s}} \rf-m}
\Big(\sum_{k=1}^m {|a|-1\choose a_1,..., a_k-1,... a_m}\Big)
R_1^{a_1}... R_m^{a_m} \chi_{2^{\s} N}
$$
$$
=\sum_{k=1}^m \sum_{\mathfrak{S}=0} ^{\lf (\frac{N}{\t})^{\frac{1}{\s}} \rf-m-k}{|a|\choose a_1,..., a_k,... a_m}R_1^{a_1}...R_k^{a_k+1}...  R_m^{a_m}
\chi_{2^{\s} N}
$$
\begin{equation}
\label{OperatorRacun}
=\sum_{k=1}^m R_k \Big(\sum_{\mathfrak{S}=0} ^{\lf (\frac{N}{\t})^{\frac{1}{\s}} \rf-m-k}{|a|\choose a_1,...,a_m}R_1^{a_1}...R_m^{a_m}
\chi_{ 2^{\s}N  }\Big),\,
\end{equation}
where for the second equality we interchange the summation and substitute $a_k$ with $a_k+1$.

Hence, for $|a|\geq 0$ we have
\begin{multline}
(I-R)w_N= \sum_{ \mathfrak{S} =0} ^{(\lf \frac{N}{\tilde \t})^{\frac{1}{\s}} \rf-m}{|a|\choose a_1,..., a_m}R_1^{a_1}...R_m^{a_m}
\chi_{2^{\s} N}\\
-\sum_{k=1}^m R_k \Big(\sum_{\mathfrak{S} =0} ^{\lf (\frac{N}{\tilde \t})^{\frac{1}{\s}} \rf-m-k}{|a|\choose a_1,...,a_m}R_1^{a_1}...R_m^{a_m}
\chi_{2^{\s} N}\\
+\sum_{\mathfrak{S} = \lf (\frac{N}{\tilde \t})^{\frac{1}{\s}} \rf-m-k+1} ^{\lf (\frac{N}{\tilde \t})^{\frac{1}{\s}} \rf-m}{|a|\choose a_1,...,a_m}R_1^{a_1}...R_m^{a_m}
\chi_{2^{\s} N}\Big) \\
=\chi_{2^{\s} N}
- \sum_{k=1}^m \sum_{\mathfrak{S} = \lf (\frac{N}{\tilde \t})^{\frac{1}{\s}} \rf-m-k+1} ^{\lf (\frac{N}{\tilde \t})^{\frac{1}{\s}} \rf-m}{|a|\choose a_1,...,a_m}R_1^{a_1}...R_k^{a_k+1}...R_m^{a_m}
\chi_{2^{\s} N} \\
= \chi_{2^{\s} N}
- \sum_{k=1}^m \sum_{\mathfrak{S} = \lf (\frac{N}{\tilde \t})^{\frac{1}{\s}} \rf-m+1} ^{\lf (\frac{N}{\tilde \t})^{\frac{1}{\s}} \rf-m+k}{|a|\choose a_1,...,a_m}R_1^{a_1}...R_m^{a_m}
\chi_{2^{\s} N},
\end{multline}
where for the second equality we used (\ref{OperatorRacun}) and for the last one we substitute $a_k$ with $a_k-1$.

Therefore, if we set
$$
e_N(x,\xi)=\sum_{k=1}^m \sum_{\mathfrak{S}  = \lf (\frac{N}{\t})^{\frac{1}{\s}} \rf-m+1} ^{ \lf (\frac{N}{\t})^{\frac{1}{\s}} \rf-m+k}{|a|\choose a_1,...,a_m}(R_1^{a_1}...R_m^{a_m}
\chi_{2^\s N})(x,\xi),
$$
then the computation of this subsection gives the equality \eqref{fundamentalEquality},
which in turn implies the fundamental representation \eqref{ocenitiJedn}.

\subsection{Estimates for  $ D^\beta ( R_1^{a_1}...R_m^{a_m}\chi_{ 2^{\s}N  })$} \label{subsecderivatives}
Note that for $N$ large enough we have
$$
(\lf (N/{\tilde \t})^{1/\s}\rf+M)^{\s}
\leq 2^{\s-1}(N/{\tilde \t}+M^{\s})
<2^{\s}N/{\tilde \t}\,
$$
so that for $|\beta|\leq M$ the following estimate holds:
$$
\mathfrak{S}+|\beta|\leq \lf (N/\tilde\t)^{1/\s}\rf+M=\lf (N/\tilde\t)^{1/\s}+M\rf < \lf 2 (N/\tilde\t)^{1/\s}\rf  \,.
$$
Thus, for $x\in K$, $\xi \in \Gamma$, and $\mathfrak{S}\geq \lf (N/\tilde\t)^{1/\s} \rf-m$,
by using (\ref{SimpleInequality}) we obtain
\begin{eqnarray}
\nonumber
|D^{\beta} (R_1^{a_1}...R_m^{a_m}\chi_{ 2^{\s}N  })(x, \xi)|
&\leq& |\xi|^{-\mathfrak{S}} A^{\mathfrak{S}+|\beta|+1}\lf N^{1/\s} \rf^{\mathfrak{S}+|\beta|}\nonumber\\
&\leq & |\xi|^{m-\lf (N/\tilde \t)^{1/\s} \rf} A^{\lf (N/\tilde \t)^{1/\s} \rf+M+1}
\lf N^{1/\s} \rf^{\lf (N/\tilde \t)^{1/\s} \rf+M}\nonumber\\
&\leq& |\xi|^{m-\lf (N/\tilde \t)^{1/\s} \rf} C^{N+1} N^{\frac{\tilde \t^{-1/\s}}{\s}N}\,,
\end{eqnarray}
for some $C>0$, which is, after enumeration
$N\to N+2^{\s-1}\tilde \t( m+ M)^{\s}$ bounded by
\begin{multline} \nonumber
|\xi|^{m-\lf ((N+2^{\s-1}\tilde \t( m+ M)^{\s})/\tilde \t)^{1/\s} \rf}
A^{N+2^{\s-1}\tilde\t( m+ M)^{\s}+1} \\
\times (N+2^{\s-1}\tilde \t( m+ M)^{\s})^{\frac{\tilde \t^{-1/\s}}{\s}(N+2^{\s-1}\tilde \t( m+ M)^{\s})},
\end{multline}
for some $A>0$. Moreover,
\begin{eqnarray}
\label{NejednakostSaN}
\Big(\frac{N+2^{\s-1}\tilde \t( m+ M)^{\s}}{\tilde \t}\Big)^{1/\s}
&\geq& 2^{\frac{1-\s}{\s}}((N/\tilde \t)^{1/\s}+2^{\frac{\s-1}{\s}}(m+M))\nonumber\\
&=&2^{\frac{1-\s}{\s}}(N/\tilde \t)^{1/\s}+m+M\,.
\end{eqnarray}
Finally, \eqref{NejednakostSaN}, $(M.2)'$ property of $N^{\frac{\tilde \t^{-1/\s}}{\s}N}$
and Stirling's formula give the estimate
\begin{equation}
\label{OcenaProizvodaOp}
|D^{\beta} R_1^{a_1}...R_m^{a_m}\chi_{ 2^{\s}N  }(x)|
\leq |\xi|^{-\lf 2^{\frac{1-\s}{\s}}(N/\tilde \t)^{1/\s} \rf-M} C^{N+1} N!^{\frac{\tilde \t^{-1/\s}}{\s}}
\end{equation}
for some $C>0$.

\subsection{The computation of ${\mathcal F}_{x\to \eta}(w_N)(\eta,\xi)$} \label{subsecfurije}
From
$$
(R_1^{a_1}...R_m^{a_m}\chi_{ 2^{\s}N  })(x,\xi)
=
\prod_{j=1}^m p_j^{a_j}(\xi)\sum_{|\alpha|\leq \mathfrak{S}}c_{\alpha}D^{\alpha}\chi_{2^{\s} N}(x)
$$
for suitable constants $c_{\alpha}$,
it follows that
$$
{\mathcal F}_{x\to \eta}(R_1 ^{a_1}...R_m ^{a_m}\chi_{2^{\s} N})(\eta,\xi)
= \prod_{j=1}^m p_j^{a_j}(\xi)\sum_{|\alpha|\leq \mathfrak{S}}c''_{\alpha}{\eta}^{\alpha}
\widehat\chi_{2^{\s} N}(\eta),
$$
so that
\begin{multline}
{\mathcal F}_{x\to \eta}(w_N)(\eta,\xi) \nonumber
\\
=\sum_{\mathfrak{S} =0} ^{\lf (\frac{N}{{\tilde \t}})^{\frac{1}{\s}} \rf-m}{|a|\choose a_1,a_2\dots a_m}\Big(\prod_{j=1}^m p_j^{a_j}(\xi)\Big)\sum_{|\alpha|\leq \mathfrak{S}}c''_{\alpha}{\eta}^{\alpha}\widehat\chi_{2^{\s} N}(\eta).
\end{multline}
Note that the number of terms in ${\mathcal F}_{x\to \eta}(w_N)(\eta,\xi)\,$
is bounded by $C 2^{\lf (N/\t)^{1/\s} \rf}$ for some $C>0$ which does not depend on $N$.

When $|\eta|\leq \varepsilon |\xi|$, $\xi \in \Gamma$,
$ |\xi| >\lf N^{1/\s} \rf$, and $N$ sufficiently large we have
\begin{multline} \nonumber
|{\mathcal F}_{x\to \eta}(w_N)(\eta,\xi)|\leq\\
\sum_{\mathfrak{S} = 0 } ^{\lf (\frac{N}{{\tilde \t}})^{\frac{1}{\s}} \rf-m}
{|a|\choose a_1,a_2\dots a_m}\Big(\prod_{j=1}^m (|p_j(\xi)||\varepsilon\xi|^j)^{a_j}\Big)
\sum_{|\alpha|\leq \mathfrak{S}}c''_{\alpha}|\widehat\chi_{2^{\s} N}(\eta)| \\
\leq A C^{\lf(N/\t)^{1/\s}\rf}|\widehat\chi_{2^{\s} N}(\eta)|,
\end{multline}
for some $A,C>0$, and we used
$$
\prod_{j=1}^m (|p_j(\xi)||\varepsilon\xi|^j)^{a_j}\leq A\varepsilon^{\mathfrak{S}}\leq A,
\quad \xi \in \Gamma, |\xi| > \lf N^{1/\s} \rf,
$$
which follows from $\varepsilon<1$ and the fact that
$\prod_{j=1}^m (|p_j(\xi)||\xi|^j)^{a_j}$ is homogeneous of order zero.


\subsection*{Acknowledgment}
This research is supported by Ministry of Education, Science and
Technological Development of Serbia through the Project no. 174024.
\par

\end{document}